\documentclass[12pt, leqno]{amsart}
\usepackage{amsmath}
\usepackage{amssymb}
\usepackage{amsthm}
\usepackage{enumerate}
\usepackage[mathscr]{eucal}
\usepackage{xcolor}
\theoremstyle{plain}
\usepackage{tikz}
\usepackage{soul}
\usepackage[normalem]{ulem}
\newtheorem{theorem}{Theorem}[section]
\newtheorem{lemma}[theorem]{Lemma}
\newtheorem{prop}[theorem]{Proposition}
\theoremstyle{definition}
\newtheorem{definition}[theorem]{Definition}
\newtheorem{remark}[theorem]{Remark}

\newtheorem{cor}[theorem]{Corollary}
\theoremstyle{remark}
\begin{document}
	\title[ On symmetric functions and symmetric operators on Banach spaces]{ On symmetric functions and symmetric operators on Banach spaces}
	
\author[Paul, Sain and  Sohel ]{Kallol Paul, Debmalya Sain and Shamim Sohel  }
	
	\address[Paul]{Vice-Chancellor, Sadhu ram Chand Murmu University,  Jhargram,  West Bengal 721524, India.\\
	Department of Mathematics, Jadavpur University (on lien), Kolkata 700032, India.}	\email{kalloldada@gmail.com}
	
		\address[Sain]{Department of Mathematics\\ Indian Institute of Information Technology, Raichur\\ Karnataka 584135 \\INDIA}	\email{saindebmalya@gmail.com}
	
		\address[Sohel]{Department of Mathematics, Jadavpur University, Kolkata 700032, India.}	\email{shamimsohel11@gmail.com}
	
		\thanks{The third  author would like to thank  CSIR, Govt. of India, for the financial support in the form of Senior Research Fellowship under the mentorship of Prof. Kallol Paul. }

	\begin{abstract}
		We study  left symmetric and right symmetric elements  in the space $\ell_{\infty}(K, \mathbb{X}) $ of bounded functions from a non-empty set $K$ to a Banach space $\mathbb{X}.$  We prove that a non-zero element $ f \in\ell_{\infty}(K, \mathbb{X}) $ is left symmetric if and only if $f$ is zero except for an element $k_0 \in K$ and $f(k_0)$ is left symmetric in $\mathbb{X}.$  We  characterize left  symmetric elements in the space $C_0(K, \mathbb{X}),$ where $K$ is a locally compact  perfectly normal space. We also study the right symmetric elements in $\ell_{\infty}(K, \mathbb{X}).$   Furthermore, we characterize right symmetric elements in $C_0(K, \mathbb{X}),$ where $K$ is a locally compact Hausdorff space and $\mathbb{X}$ is real Banach space.   As an application of the results obtained in this article, we  characterize the left symmetric and right symmetric operators on some special Banach spaces. These results improve and generalize the existing ones on the study of left and right symmetric elements in operator spaces. 
	\end{abstract}
	
\subjclass[2020]{Primary 47L05, Secondary 46B20}
\keywords{Birkhoff-James orthogonality, left symmetric points, right symmetric points, space of bounded functions, space of continuous functions }

	\maketitle
	\section{Introduction}

	In the realm of  Banach space geometry, the concept of symmetric points plays an important role. Unlike Hilbert spaces, the Birkhoff-James orthogonality is not symmetric in general Banach spaces. In view of this, two new notions of  left symmetric points and right symmetric points have been introduced \cite{S1}. Several mathematicians have studied the symmetric points and their applications in the isometric theory of Banach spaces in \cite{CSS,GSP, PMW,S1, SRBB, T} to gain insights into  the geometric aspects of the concerned spaces, including the study of onto isometries.  In this  article, we study the symmetric points in $\ell_{\infty}(K, \mathbb{X}),$ the space of all bounded functions from $K$ to $\mathbb{X}$ and $C(K, \mathbb{X}),$ the space of all continuous functions from $K$ to $\mathbb{X}$, where $\mathbb{X}$ is a Banach space and $K$ is a non-empty set.\\
	
	We use the symbols $ \mathbb{X}, \mathbb{Y}$ to denote Banach spaces  over the field  $\mathbb{K},$ where  $\mathbb{K}$ is either the complex field $\mathbb{C}$ or the real field $\mathbb{R}.$  Let $ B_{\mathbb{X}}= \{ x \in \mathbb{X}: \|x\| \leq 1\} $  and $ S_{\mathbb{X}}= \{ x \in \mathbb{X}: \|x\| = 1\} $ denote  the unit ball and the unit sphere of $\mathbb{X},$ respectively. The  dual space of $ \mathbb{X}$ is denoted by  $ \mathbb{X}^*.$  Let $ \mathbb{L}(\mathbb{X}, \mathbb{Y}) $ $(\mathbb{K}(\mathbb{X}, \mathbb{Y}) )$ denote the Banach space of all bounded (compact) linear operators  from $\mathbb{X}$ to $\mathbb{Y}.$ The convex hull of a non-empty set $ S  \subset \mathbb{X}$  is the intersection of all convex sets in $\mathbb{X}$ containing $S$ and it is denoted by $co (S).$  For a non-empty convex set $A,$ an element $z \in A$ is said to be an extreme point of $ A $ if the equality $z = (1 - t)x + t y$, with $t \in (0, 1)$ and $x, y \in A,$ implies that $ x = y = z.$ The set of all extreme points of $A$ is denoted by $Ext(A).$ A Banach space is said to be strictly convex if every point of $S_{\mathbb{X}}$ is an extreme point of $B_{\mathbb{X}}.$
	For  $T \in \mathbb{L}(\mathbb{X}, \mathbb{Y}),$ the norm attainment set of $T,$  denoted by $M_T,$ is defined as $M_T=\{x \in S_{\mathbb{X}}: \|Tx\|=\|T\|\}.$ 	Given any non-zero $ x \in \mathbb{X},$ $ f \in S_{\mathbb{X}^*} $ is said to be a support functional at $x$  if $ f(x)= \|x\|.$ Let $ J(x) := \{ f \in S_{\mathbb{X}^*}: f(x)= \|x\|\}.$  Given a subset $M$ of $\mathbb{X}^*,$ $\overline{M}^{w^*}$ denotes the closure of $M$ with respect to the weak*- topology defined on $\mathbb{X}^*.$

		Let us now recall from  \cite{B, J} that  an element $x \in \mathbb{X}$ is said to be Birkhoff-James orthogonal to $y \in \mathbb{X}$ if $ \| x + \lambda y \| \geq \|x\|,$ for all $\lambda \in \mathbb{K}.$   Symbolically, it is written as $ x \perp_B y.$ For real Banach spaces, this orthogonality relation is  dissected in two parts as $x^+$ and $x^-,$ in \cite{S2}. Given $x, y \in \mathbb{X},$ we say that $y \in x^+$   if $\|x+ \lambda y\| \geq \|x\|,$ for any $\lambda \geq 0$ and $y \in x^-$ if $\|x+ \lambda y\| \geq \|x\|,$ for any $\lambda \leq 0.$ For any $x, y \in \mathbb{X},$ it is easy to verify that either $y \in x^+$ or $y \in x^-.$ Approximate versions of these concepts have been introduced in \cite{SPM}, for further analyzing the geometry of Banach spaces as follows :  For $x, y \in \mathbb{X}$ and $\epsilon \in [0,1),$ we say that $y \in x^{+\epsilon}$ if $\|x+ \lambda y\| \geq \sqrt{1- \epsilon^2} \|x\|, \forall \lambda \geq 0$ and  $y \in x^{-\epsilon}$ if $\|x+ \lambda y\| \geq \sqrt{1- \epsilon^2} \|x\|, \forall \lambda \leq 0.$
		
			   It is clear that in general, Birkhoff-James orthogonality is  not symmetric. In this context the notions of left and right symmetric  points were introduced in \cite{S1}:
		  
	\begin{definition}
		An element $x \in \mathbb{X}$ is said to be a left symmetric point if $x \perp_B y$ implies
		that $y \perp_B x$ for all $y \in \mathbb{X}.$ Similarly, a point $ x \in \mathbb{X}$ is said to be a right symmetric point if
		$y \perp_B x$ implies that $x \perp_B y$ for all $y \in  \mathbb{X}.$ An element $x \in \mathbb{X}$ is said to be a symmetric point if $x$ is both left symmetric and right symmetric.   
	\end{definition}

While complete characterizations have been obtained for left symmetric points \cite[Th. 2.1]{SRBB} and right symmetric points \cite[Th. 2.2]{SRBB} in real Banach spaces, explicitly determining these points still remains a challenging task in the space of bounded (continuous) functions. Of course, this difficulty is further elevated in the complex case. This article aims to address this problem by determining the explicit forms of left (right) symmetric elements in the bounded (continuous) function spaces.

	For  a non-empty set $K$ and a Banach space $\mathbb{X},$ the Banach space of all bounded functions defined from $K$ to $\mathbb{X},$  endowed with the supremum norm, is denoted by $\ell_{\infty}(K, \mathbb{X}).$ Given a compact Hausdorff topological space $K$ and a Banach space $\mathbb{X}$, we write $C(K, \mathbb{X} )$ to denote the Banach space of all bounded continuous functions from $K$ to $\mathbb{X}$, endowed with the supremum norm. Clearly, $C(K,\mathbb{X})$ is  embedded into  $\ell_{\infty}(K, \mathbb{X}).$ Whenever $K$ is finite, it follows trivially that $\ell_{\infty}(K, \mathbb{X}) = C(K, \mathbb{X}).$ For the sake of simplicity, if $K$ is a finite set and $|K|=n,$ then $\ell_{\infty}(K, \mathbb{X})$ is denoted as $	\ell_{\infty}^n(\mathbb{X}).$ Given a locally compact Hausdorff space $K$ and a Banach space $\mathbb{X}$, the space $C_0(K, \mathbb{X})$ is the space of all bounded continuous function $f$ having the property that for any $\epsilon >0,$ there exists a compact set $\Gamma \subset K$ such that $\|f(k)\|< \epsilon,$ for any $k \in K\setminus \Gamma.$ So, whenever $K$ is compact, $C_0(K, \mathbb{X})= C(K, \mathbb{X}).$
		 For a function $f \in C_0(K, \mathbb{X}),$ the norm attainment set of $f,$ denoted by $M_f,$ is defined as $	M_f= \{k\in K: \|f(k)\|= \|f\|\}.$ Whenever $\mathbb{X}= \mathbb{R}$ or $ \mathbb{C},$ we use the standard    notations  $C(K)$ and $C(K_0)$  in place of $C(K, \mathbb{X})$ and $C_0(K, \mathbb{X}),$ respectively.\\

	This article is divided into three sections including the introductory one.  In the following preliminary section we study some basic results on Birkhoff-James orthogonality in the function spaces. The main results of this article are further divided into two sections. In Section-I,  we study the symmetric points in the spaces $\ell_{\infty}(K, \mathbb{X}), C(K, \mathbb{X}),$ and $C_0(K, \mathbb{X}).$ We characterize the left symmetric points in the space $\ell_{\infty}(K, \mathbb{X})$ and provide a necessary condition for right symmetric points. We also obtain a complete characterization of left symmetric functions in the space $C_0(K, \mathbb{X}),$ under the condition that $K$ is locally compact and perfectly normal. Moreover, we characterize the right symmetric points in $C(K, \mathbb{X}),$ under the assumption that $K$ is a compact Hausdorff space. As an application of our study, we provide a necessary condition for onto linear isometries in the space $\ell_{\infty}^n(\mathbb{X}).$ In Section-II,  we study the symmetric points in the space of bounded linear operators. Using the results obtained in Section-I, we characterize the symmetric points for the spaces $\mathbb{K}(\mathbb{X}, C(K)).$ We also present refinements of some known results on symmetric points in the spaces $\mathbb{L}(\ell_{1}^n, \mathbb{X}) $ and $\mathbb{L}(\mathbb{X}, \ell_{\infty}^n).$

	\section{Preliminaries}

	The following characterizations of Birkhoff-James orthogonality and relevant results  in terms of support functionals will be used in many places in this article.
	
	\begin{theorem}\cite{J}\label{J}
		Let $\mathbb{X}$ be a  Banach space and let $ x, y \in \mathbb{X}.$ Then $x \perp_B y$ if and only if there exists $f \in \mathbb{X}^*$ such that $f(x)=\|f\| \|x\|$ and $ f(y)=0. $
	\end{theorem}
	
	\begin{lemma}\label{functional}\cite[Lemma. 7.3.2]{MPS}
		Let $\mathbb{X}$ be a real Banach space and $x, u \in \mathbb{X}.$ Then the following holds true:
		\begin{itemize}
			\item[(i)] $u \in x^+ \setminus x^-$ if and only if $f(u) > 0,$ $\forall f \in J(x).$
			\item[(ii)] $u \in x^- \setminus x^+$ if and only if $f(u) < 0,$ $\forall f \in J(x)$.
		\end{itemize}
	\end{lemma}

In \cite{MMQRS}, Birkhoff-James orthogonality has been studied in the space $\ell_{\infty}(K, \mathbb{X})$ and $C(K, \mathbb{X}),$ for some non-empty set $K$ and a Banach space $\mathbb{X}.$ In order to describe the symmetric points in these spaces, we need the following characterizations of Birkhoff-James orthogonality that follow directly from \cite[Th. 3.2, Th. 3.5]{MMQRS}.

	\begin{theorem}\label{theorem:martin}\cite[Th 3.2]{MMQRS}
		Let $K$ be a non-empty set and let $\mathbb{X}$ be a Banach space. Suppose $C \subset S_{\mathbb{X}^*}$ be such that $B_{\mathbb{X}^*}= \overline{co(C)}^{w^*} .$ Then for $f, g \in \ell_{\infty}(K, \mathbb{X} ),$  $ f \perp_B g$ if and only if 
		\[
		0 \in co\bigg(\bigg\{ \lim y_n^*(g(k_n)): k_n \in K, y_n^* \in C, \forall n \in \mathbb{N}, \lim y_n^*(f(k_n))= \|f\|\bigg\}\bigg).
		\]
	\end{theorem}
	
		\begin{cor}\label{martin:continuous}\cite[Th 3.5]{MMQRS}
		Let $K$ be a compact Hausdorff space and let $\mathbb{X}$ be a Banach space. Suppose let $C \subset S_{\mathbb{X}^*}$ be such that $B_{\mathbb{X}^*}= \overline{co(C)}^{w^*}.$  Then  for $f, g \in C(K, \mathbb{X} ),$  $f \perp_B g$ if and only if 
		\[
		0 \in co\bigg(\bigg\{  y^*(g(k)): k \in K, y^* \in C,   y^*(f(k))= \|f\|\bigg\}\bigg).
		\]
	\end{cor}

Observe that $B_{\mathbb{X}^*}$ being weak*-compact and convex, it follows from the Krein-Milman Theorem that $B_{\mathbb{X}^*}= \overline{co(Ext(B_{\mathbb{X}^*}))}^{w*}.$ Thus, both Theorem \ref{theorem:martin} and  Corollary \ref{martin:continuous} remain valid if we take $C= Ext(B_{\mathbb{X}^*}).$\\

 %Throughout this article the set $Ext(B_{\mathbb{X}^*})$ will be used in place of $C$ when we use the above two results. 
 In the same spirit, we study orthogonality in the space $C_0(K, \mathbb{X}),$ when $K$ is locally compact Hausdorff. For this purpose, we require the following observations.

	\begin{prop}\label{compact}
		Let $K$ be a locally compact Hausdorff space and let $\mathbb{X}$ be a Banach space. Then for any $f \in S_{C_0(K, \mathbb{X})},$ $ M_f$ is non-empty and compact.
	\end{prop}

	\begin{proof}
		Let $\epsilon= \frac{1}{2}.$ Suppose that $\Gamma \subset K$ is a compact set such that $\|f(k)\|< \epsilon,$  $\forall k \in K \setminus \Gamma.$ Clearly, 
	\begin{eqnarray*}
		\|f\|= \sup_{k \in K} \|f(k)\|= \sup_{k \in \Gamma} \|f(k)\|.
	\end{eqnarray*}
As $\Gamma$ is compact, there exists $k' \in \Gamma$ such that $\|f(k')\|= \sup_{k \in \Gamma} \|f(k)\|= \|f\|.$ Hence $k' \in M_f$ and $M_f$ is non-empty.
			Let $k_\alpha$ be a net in $M_f.$ Clearly, $\{k_\alpha\} \subset \Gamma.$ As $\Gamma$ is compact, there exists a subnet $\{k_{\alpha_\lambda}\}$ of $\{k_{\alpha}\}$ such that  $k_{\alpha_\lambda} \to k_0 \in \Gamma.$ To prove that $M_f$ is compact, it suffices to show that $ k_0 \in M_f.$ Since $f$ is continuous, it follows that $f(k_{\alpha_\lambda}) \to f(k_0)$ and $\|f(k_{\alpha_\lambda})\| \to \|f(k_0)\|.$ As $\{k_{\alpha_\lambda}\} \in M_f,$ we obtain that $\|f(k_0)\|=1.$
	\end{proof}

	\begin{theorem}\label{martin:continuous2}
	Let $K$ be a locally compact Hausdorff space and  let $\mathbb{X}$ be a Banach space. Suppose $f, g \in C_0(K, \mathbb{X} ).$ Then $ 	f \perp_B g  $  if and only if 
	\[
0 \in co\bigg(\bigg\{  y^*(g(k)): k \in K, y^* \in Ext(B_{\mathbb{X}^*}),   y^*(f(k))= \|f\|\bigg\}\bigg).
	\]
\end{theorem}

\begin{proof}
	From Proposition \ref{compact}, for any $f \in C_0(K, \mathbb{X})$, $M_f$ is non-empty and  compact. Now the proof of this theorem follows in the same line of the proof of \cite[Th. 3.5]{MMQRS}. 
\end{proof}

For a real Banach space $\mathbb{X},$ the following result can be deduced easily by applying Corollary \ref{martin:continuous}.

\begin{cor}\label{ortho:continuous}
	Let $K$ be a compact Hausdorff space and $\mathbb{X}$ be a real Banach space. Let $f, g \in C(K, \mathbb{X}).$ Then $f \perp_B g$ if and only if there exists $k_1, k_2 \in M_f$ such that $g(k_1) \in f(k_1)^+$ and $g(k_2) \in f(k_2)^-.$  
\end{cor}

We end this section with the following important property of 
a locally compact Hausdorff space  which will be used later on. 

\begin{theorem}\label{locally compact hausdorff}\cite[Th. 2.7]{R}
Suppose $U$ is open in a locally compact Hausdorff space $K,$ $S \subset U$ and $S$ is compact. Then there is an open set $V$ with compact closure $\overline{V}$ such that $S \subset V \subset \overline{V} \subset U.$ 
\end{theorem}

%Although the above result is a complete characterization of right symmetric points in real Banach spaces, it does not provide an explicit description. In this article we provide explicit descriptions of left(right) symmetric functions and operators defined on some special Banach spaces.

\section{Main Results.}

\section*{Section-I}

	We begin with the characterization of left symmetric elements in the space $\ell_{\infty}(K, \mathbb{X}).$	
	
	\begin{theorem}\label{left}
		Let $K$ be any non-empty set and let $\mathbb{X}$ be a Banach space. Then $f \in S_{\ell_{\infty}(K, \mathbb{X})}$ is left symmetric  if and only if $f$ satisfies the following: 
		\begin{itemize}
			\item[(i)] there exists a  $k_0 \in K$ such that $f(k_0) \in S_{\mathbb{X}}$ and $f(k)=0, \forall k_0 \neq k \in K .$
			
			\item[(ii)] $f(k_0)$ is  left symmetric.
		\end{itemize}   
	\end{theorem}

	\begin{proof}
	To prove the sufficient part, let us assume that $f \in S_{\ell_{\infty}(K, \mathbb{X})}$  and $f \perp_B g.$ We prove that $g \perp_B f.$  Let $\{k_n\} \subset K$ such that $\lim \|g(k_n)\|=\|g\|.$ Consider the following two cases:
		
		Case 1: Let $k_n \neq k_0,$ for all but finitely many $n \in \mathbb{N}.$ Consider $y_n^* \in Ext(B_{\mathbb{X}^*})$ such that $\lim y_n^*(g(k_n))= \|g\|.$ Since $f(k_n)=0,$ for all but finitely many $n \in \mathbb{N},$ we obtain that 
		\[
		0 \in co(\{\lim y_n^*(f(k_n)): k_n \in K, y_n^* \in Ext(B_{\mathbb{X}^*}), \forall n \in \mathbb{N}, \lim y_n^*(g(k_n))= \|g\|\}).
		\]
		Now using Theorem \ref{theorem:martin}, we obtain that $g \perp_B f.$
		
	Case 2: Suppose that the sequence $\{k_n\}$ has a subsequence $\{k_{n_r}\}$ such that $k_{n_r}= k_0,\forall r \in \mathbb{N}.$ Then  $\|g\|=\|g(k_0)\|.$ As $f \perp_B g$ and $f(k)=0, \forall k \in K \setminus \{k_0\},$ so for any scalar $\lambda,$
		\[
		\sup\{\|f(k_0)+ \lambda g(k_0)\|, |\lambda| \|g(k)\|: k \in K \setminus \{k_0\}\}= \|f+ \lambda g\| \geq \|f\|= \|f(k_0)\|=1.
		\]
		Let $|\lambda|< \frac{1}{2\|g\|},$ then $|\lambda|\|g(k)\|< \frac{1}{2}.$ So, whenever $|\lambda|< \frac{1}{2\|g\|},$ 
		\[
		\|f(k_0) + \lambda g(k_0)\| \geq \     |f(k_0)\|.
		\]
		From the convexity of the norm, it is easy to observe that for any scalar $\lambda,$ $\|f(k_0) + \lambda g(k_0)\| \geq \|f(k_0)\|.$ So, $f(k_0) \perp_B g(k_0). $	 Since $f(k_0)$ is left symmetric, $g(k_0) \perp_B f(k_0).$ Observe that for any scalar  $\lambda,$ $$ \|g + \lambda f\| \geq \|g(k_0) + \lambda f(k_0)\| \geq \|g(k_0)\|=\|g\|, $$ which is equivalent to $g \perp_B f.$ Thus, $f$ is left symmetric.

	To prove the necessary part, we first show  that $f$ satisfies (i).  Suppose on the contrary that  there exist $k_1, k_2 \in K$ such that $f(k_1) \neq 0 \neq f(k_2),$ where $k_1 \neq k_2$. We choose $\{k_n\}_{n \in \mathbb{N}} \subset  K$   such that      $\lim \|f(k_n)\|=\|f\|.$ %Without loss of generality we assume that $f(k_n) \neq 0,$ for any $n \in \mathbb{N}.$
	  Define $g : K \to \mathbb{X}$ such that  $g(k_1)=\frac{f(k_1)}{\|f(k_1)\|},$  and $g(k)=0,$ for any $k \in K \setminus \{k_1\}.$ Clearly, $\|g\|=1.$ Then for any $y_n^* \in Ext(B_{\mathbb{X}^*})$ with $\lim y_n^*(f(k_n))=\|f\|,$ we obtain  that $y_n^*(g(k_n))=0, \forall n > 1.$ Therefore, 
			$$ 0 \in co(\{ \lim y_n^*(g(k_n)): k_n \in K, y_n^* \in Ext(B_{\mathbb{X}^*}), \forall n \in \mathbb{N}, \lim y_n^*(f(k_n))= \|f\|\}) . $$ So, using Theorem \ref{theorem:martin}, we obtain that $f \perp_B g.$ 
Next we show that $g \not \perp_B f.$			
Let $\{k_n'\} \subset   K$ and $ \{y_n^*\} \subset  Ext(B_{\mathbb{X}^*})$ such that $\lim y_n^*(g(k_n'))=1,$  $ \forall n \in \mathbb{N}.$ Note that $g(k_1) \in S_{\mathbb{X}}$ and $g(k)=0, \forall k \neq k_1.$ As $\lim y_n^*(g(k_n'))=1,$ without loss of generality we assume that $k_n'= k_1,\forall n \in \mathbb{N}.$ As $g(k_1)=  \frac{f(k_1)}{\|f(k_1)\|},$ it follows that $\lim y_n^*(g(k_1))=1. $ Thus, $ \lim y_n^*(f(k_1)) = \|f(k_1)\| \neq 0.$ So, 
\[
0 \notin co(\{\lim y_n^*(f(k_1)): y_n^* \in Ext(B_{\mathbb{X}^*}),\lim y_n^*(g(k_1))=1  \}),
\]
which implies that $g \not \perp_B f.$
This contradicts that $f$ is left symmetric. 
		Therefore, $f$ satisfies (i).
		
		Next we  show that (ii) holds. Suppose on the contrary that $f(k_0)$ is not  left symmetric. Then there exists $w_0 \in S_{\mathbb{X}}$ such that $f(k_0) \perp_B w_0$ and $w_0 \not\perp_B f(k_0).$ So, there exists $\lambda_0 \in \mathbb{K}$ such that $\|w_0 + \lambda_0 f(k_0)\|< \|w_0\|.$ Define $g: K \to \mathbb{X}$ such that $g(k_0)=w_0$ and $g(k)=0,$ for any  $k \in K \setminus\{k_0\}.$ Clearly, $\|g\|=1.$ As $f$ satisfies (i), it follows that for any scalar $\lambda $, $$\|f + \lambda g\| = \|f(k_0)+ \lambda g(k_0)\|= \|f(k_0)+ \lambda w_0\| \geq \|f(k_0)\|.$$ So, $f \perp_B g.$  We note that $ f(k) = 0 \, \, \forall (k_0 \neq) k \in K $  and moreover,
		$$ \|g + \lambda_0 f\|=\|g(k_0) + \lambda_0 f(k_0)\|= \|w_0 + \lambda_0         f(k_0)\|< \|w_0\|=1=\|g\|,$$ 
		which implies that $g \not\perp_B f.$ This contradicts the fact that $f$ is left symmetric. Thus $f(k_0)$ is left symmetric. \\
	\end{proof}
	
	As an immediate corollary of the above theorem we obtain the following observation.

	\begin{cor}
		Let $K$ be a Hausdorff topological  space and $\mathbb{X}$ be a Banach space. Then   there is a non-zero left symmetric continuous function in $\ell_{\infty}(K, \mathbb{X})$ if and only if $K$ contains an isolated point and $\mathbb{X}$ contains a non-zero left symmetric point.
	\end{cor}

\begin{proof}
	
		Let $f$ be a non-zero left symmetric continuous function in $\ell_{\infty}(K, \mathbb{X}).$ Without loss of generality assume that $\|f\|=1.$ From Theorem \ref{left}, there exists $k_0 \in K$ and a left symmetric point $w \in S_{\mathbb{X}}$ of $\mathbb{X}$ such that $f(k_0)=w$ and $f(k)=0, \forall k \in K \setminus \{k_0\}.$ Let $U= \{u \in \mathbb{X}: \|u -w\|< \frac{1}{2}\}.$ Clearly, $U$ is open in $\mathbb{X}.$ As $f$ is continuous, $f^{-1}(U)$ is open. Observe that $f^{-1}(U)= \{k_0\}.$ Therefore, $k_0$ is an isolated point.

Conversely,	let $w\in S_{\mathbb{X}}$ be a left symmetric point of $\mathbb{X}$ and let  $k_0$ be an isolated point of $K.$ Define $f : K \to \mathbb{X}$ such that $f(k_0)=w$ and $f(k)=0, \forall K \setminus\{k_0\}.$ Clearly $f $ is continuous and it follows from Theorem \ref{left} that $f$  is left symmetric.
	\end{proof}	
	
%	\begin{cor}
%		Let $K$ be a discrete space and let $\mathbb{X} $ be a  Banach space. Then $f \in C(K, \mathbb{X})$ is left symmetric if and only if $f$ satisfies  the following: 
%		\begin{itemize}
%			\item[(i)] there exists exactly one point $k_0 \in K$ such that $f(k_0) \in S_{\mathbb{X}}$ and $f(k)=0,$ for any $k \in K \setminus \{k_0\}$
			
%			\item[(ii)] $f(k_0)$ is  left symmetric.\\
%		\end{itemize}
%	\end{cor}

 As the space $C_0(K, \mathbb{X})$ is a subspace of $\ell_{\infty}(K, \mathbb{X}),$ if a continuous function $f$ satisfies the sufficient condition of Theorem \ref{left}, then $f$ is also  left symmetric  in $C_0(K, \mathbb{X}).$ In the next theorem, we show that the necessary part also holds true for left symmetric point in $C_0(K, \mathbb{X}),$ under the condition that $K$ is  perfectly normal. Let us recall that a topological space $K$ is said to be perfectly normal if  $K$ is normal and every closed set of $K$ is a $G_\delta$ set.
	
	\begin{theorem}\label{perfect}
		Let $K$ be a locally compact, perfectly normal space and let $\mathbb{X}$ be a Banach space. Then   $f \in S_{C_0(K, \mathbb{X})}$  is left symmetric if and only if $f$ satisfies the following : 
		\begin{itemize}
			\item[(i)] there exists  $k_0 \in K$ such that $f(k_0) \in S_{\mathbb{X}}$ and  $f(k)=0,$ $\forall k_0 \neq k \in K .$
				\item[(ii)] $f(k_0) $ is left symmetric.
		\end{itemize} 
		
	\end{theorem}
	
	\begin{proof}
		Since the sufficient part follows from Theorem \ref{left}, we prove only the necessary part.

		(i) Suppose on the contrary that there exist $k_0, k_1 \in K$ such that $f(k_0) \neq 0 \neq f(k_1).$ From Proposition \ref{compact}, it follows that $M_f$ is non-empty. Without loss of generality we assume that $f(k_0) \in S_{\mathbb{X}}.$ Let $v \in S_{\mathbb{X}}$ be such that $v \not\perp_B f(k_1).$ As $K$ is locally compact and Hausdorff, Theorem \ref{locally compact hausdorff} ensures that there exists an open set $U$  containing $k_1$ such that $k_0 \notin U$ and $\overline{U}$ is a compact set. As $K$ is perfectly normal and $\{k_1\},   K \setminus U $ are two disjoint closed sets in $K$, so by the Urysohn's lemma (see \cite{Munkres}), there exists a continuous function $h: K \to [0,1]$ such that $h^{-1}(\{1\})= \{k_1\}$ and $h^{-1}(\{0\})= K \setminus U.$ Define $g: K \to \mathbb{X}$ as  $$g(k)= h(k)v  ,~ \forall k \in K.$$
			 Clearly, $g$ is continuous.
			 %  Take $\epsilon >0.$
			 % Let $\epsilon'= \|f(k_1)\|\epsilon.$ 
	%		 Since $f \in C_0(K, \mathbb{X}),$ there exists a compact set $\Gamma$ of $K$ such that $\|f(k)\|< \epsilon,$ $\forall k \in K \setminus \Gamma.$ Observe that for any $k \in K \setminus \Gamma,$ 	 $$ \|g(k)\|= \|h(k)f(k)\| \leq \|f(k)\|< \epsilon.$$	As $K$ is locally compact and Hausdorff, 
	As for any $k \in K \setminus \overline{U},$ $\|g(k)\|=0,$ it is clear that $g \in C_0(K, \mathbb{X}).$
%	Clearly, $\|g\|= \sup_{k \in K} \|g(k)\|= \sup_{k \in K} |h(k)|\|f(k)\|\leq \|f\|.$	
   As $h^{-1} (\{1\})=\{k_1\}, h^{-1}(\{0\})= K \setminus U,$ so  we have $g(k_1)=v,$   and  $g(k)=0,$ $\forall k \in K \setminus U.$ Moreover, $\|g\|=1$ and $M_g=\{k_1\}.$ Now observe that for any scalar $\lambda,$ 
		\[
		\|f + \lambda g\|= \sup_{k \in K}\|f(k)+\lambda g(k)\| \geq \|f(k_0)+ \lambda g(k_0)\|=\|f(k_0)\|=1=\|f\|.
		\]
		Therefore, $f \perp_B g.$
		Also  observe that  $v \not\perp_B f(k_1).$ Using Theorem  \ref{J}, there exists no $y^* \in J(v)$ such that $y^*(f(k_1))=0.$ As $J(v)$ is convex, we deduce that 
		$$0 \notin co(\{y^*(f(k_1)): y^* \in J(v)\}).$$
		 As $M_g=\{k_1\}, $ it is now immediate that 
		$$0 \notin co(\{y^*(f(k_1)): y^* \in Ext(B_{\mathbb{X}^*}), y^*(g(k_1))=\|g\|\}).$$ 
		Therefore,  using Theorem  \ref{martin:continuous2}, $g \not\perp_B f.$ This contradicts that $f$ is  left symmetric.
		
		(ii) Suppose on the contrary that $f(k_0)$ is not left symmetric. So, there exists $v \in S_{\mathbb{X}}$ such that $ f(k_0) \perp_B v$ but $v \not\perp_B f(k_0).$ Let $U$ be an open set containing $k_0$ such that $\overline{U}$ is compact.  As $K$ is perfectly normal, there exists a continuous function $h: K \to [0,1]$ such that $h^{-1}(\{1\})= \{k_1\}$ and $h^{-1}(\{0\})= K \setminus U.$ Define $g: K \to \mathbb{X}$ such that $$g(k)= h(k)v  ,~ \forall k \in K.$$ 
	 Proceeding similarly  as in the proof of (i), we get $g \in C_0(K, \mathbb{X})$
	 and $M_g = \{k_0\}, g(k_0)=v.$  As $f(k_0) \perp_B v$ and $g(k_0)=v,$ observe that for any scalar $\lambda,$ 
		\[
		\|f + \lambda g\|= \sup_{k \in K}\|f(k)+\lambda g(k)\| \geq \|f(k_0)+ \lambda g(k_0)\| = \|f(k_0)+ \lambda v\|\geq \|f(k_0)\|=1=\|f\|.
		\]
		Therefore, $f \perp_B g.$  Also  observe that  $v \not\perp_B f(k_0)$ and following  similar arguments as  in the proof of the first part of the theorem, we obtain that $g \not\perp_B f.$ This contradicts the fact  that $f$ is left symmetric. 
	\end{proof}

We next provide a necessary condition for a right symmetric element in the space $\ell_{\infty}(K, \mathbb{X}).$
	
	\begin{theorem}\label{right}
		Let $K$ be any non-empty set and $\mathbb{X}$ be a  Banach space. Suppose that $f \in S_{\ell_{\infty}(K, \mathbb{X})}$ is right symmetric. Then
		\begin{itemize}
			\item[(i)] $f(k) \in S_{\mathbb{X}},$  $\forall k \in K.$
			\item[(ii)] 	$f(k)$ is right symmetric,  $\forall k \in K.$
		\end{itemize} 
	\end{theorem}
	
	\begin{proof}
		Let us first prove (i).  Suppose on the contrary that there exists $k_o \in K$ such that $f(k_0) \notin S_{\mathbb{X}}.$ Take $w_0 \in S_{\mathbb{X}}$ such that $w_0 \perp_B f(k_0) .$ Define $g : K \to \mathbb{X}$ such that $g(k_0)= w_0$ and $g(k)= f(k),$ $\forall k_0 \neq k \in K.$ Clearly, $\|g\|=1.$ 	Now observe that for any scalar $\lambda,$ 
		$$ \|g + \lambda f\| \geq \|g(k_0)+ \lambda f(k_0)\|=\|w_0+ \lambda f(k_0)\| \geq \|w_0\|=1= \|g\|.$$ 
		So,  $g \perp_B f.$  We claim that $f \not \perp_B g,$ i.e., 
			\begin{eqnarray*}
			 0 \notin co \big( \big\{\lim y_n^*(g(k_n)): y_n^* \in Ext(B_{\mathbb{Y}^*}), k_n \in K, \forall n \in \mathbb{N}, \lim y_n^*(f(k_n))=\|f\|\big\} \big).
			\end{eqnarray*}
			Consider $k_n \in K$ and $y_n^* \in Ext(B_{\mathbb{Y}^*})$ such that $\lim y_n^*(f(k_n))= \|f\|=1.$	Then there are two possibilities for the sequence $\{k_n\}$:   (a) $\{k_n\}$ has a subsequence $\{k_{n_r}\}$ such that $k_{n_r}= k_0, \forall r \in \mathbb{N}$ and (b) $k_n \neq k_0,$ for all but finitely many $n.$ Case (a) is not possible as $f(k_0) \notin S_{\mathbb{X}}$ and so, $|y_{n_r}^*(f(k_{n_r}))|= |y_{n_r}^*(f(k_0))|\leq \|f(k_0)\|< 1,$ for all $r.$ For Case (b), we get $g(k_n)= f(k_n), $  for all but finitely many $n$ and so, $\lim y_n^*(g(k_n))= \lim y_n^*(f(k_n))=1.$ This shows that 
			$$ 0 \notin co \bigg( \bigg\{\lim y_n^*(g(k_n)): y_n^* \in Ext(B_{\mathbb{Y}^*}), k_n \in K, \forall n \in \mathbb{N}, \lim y_n^*(f(k_n))=\|f\|\bigg\} \bigg).$$  
	Thus $g \perp_B f$ and $f \not \perp_B g,$ which contradicts that $f$ is right symmetric.	 Hence $f(k) \in S_{\mathbb{X}},$ for each $k \in K.$

		Let us now prove (ii).
		Suppose on the contrary that there exists $k_0 \in K$ such that $f(k_0)$ is not right symmetric.  Take $w_0 \in S_\mathbb{X}$ such that $w_0 \perp_B f(k_0) $ but $f(k_0) \not\perp_B w_0.$ Then there exists scalar $\lambda_0$ such that $\|f(k_0)+ \lambda_0 w_0\|< \|f(k_0)\|=1.$ Using the convexity of the norm, it is easy to observe that for any $0 < t \leq 1,$ $\|f(k_0)+ t\lambda_0 w_0\|< \|f(k_0)\|=1.$ Take $\mu = t_0 \lambda_0$ such that $|\mu|<1,$ where $0 < t_0 \leq 1.$ Define $g: K \to \mathbb{X}$ such that $g(k_0)=w_0$ and $g(k)= - \overline{\mu} f(k),$ for any $k \in K \setminus \{k_0\}.$   Clearly, $\|g\|=1.$ 	As $w_0 \perp_B f(k_0),$	we now observe that for any scalar  $\lambda,$ 
		$$ \|g + \lambda f\| \geq \|g(k_0)+ \lambda f(k_0)\|=\|w_0+ \lambda f(k_0)\| \geq \|w_0\|=1= \|g\|.$$ 
		So,  $g \perp_B f.$  On the other hand, 
			\begin{eqnarray*}
			\|f+ \mu g\| & =&  \sup_{k \in K} \|f(k)+ \mu g(k)\| \\
			&=&  \sup\bigg\{\|f(k) - \mu \overline{\mu} f(k) \|, \|f(k_0)+ \mu g(k_0)\|: k \in K \setminus \{k_0\}\bigg\} \\ 
			&=&  \sup \bigg\{|	(1-|\mu|^2)|\|f(k)\|, \|f(k_0)+\mu w_0\| : k \in K \setminus \{k_0\}\bigg\} \\& < & \max \{(1 -| \mu|^2), 1\} =1= \|f\|.
		\end{eqnarray*} 
		Therefore, $f \not \perp_B g.$ This contradicts the fact  that $f$ is right symmetric, and completes the proof.
		\end{proof}

As an immediate consequence of Theorem \ref{right}, we obtain a necessary condition for a continuous function to be right symmetric in $\ell_{\infty}(K, \mathbb{X}),$ when $K$ is a connected Hausdorff space.

	\begin{cor}
		Let $K$ be a connected Hausdorff space and let $\mathbb{X}$ be a Banach space such that the set of all right symmetric points in $S_{\mathbb{X}}$ is finite. Suppose  $f \in S_{\ell_{\infty}(K, \mathbb{X})}$ is continuous and  right symmetric in $\ell_{\infty}(K, \mathbb{X}).$ Then 
	 there exists a right symmetric point $x_0 \in S_{\mathbb{X}}$ in $\mathbb{X}$ such that $f(k)=x_0,$ $\forall k \in K.$
		\end{cor}
	
%%	\begin{cor}Let $K$ be a  Hausdorff space and let $\mathbb{X}$ be a Banach space such that the set of all right symmetric points in $S_{\mathbb{X}}$ is finite. Then $K$ is connected if and only if the following are equivalent:\begin{itemize}	\item[(i)] $f$ is right symmetric continuous function in $\ell_{\infty}(K, \mathbb{X})$	\item[(ii)] there exists a right symmetric point $x_0 \in S_{\mathbb{X}}$ in $\mathbb{X}$ such that $f(k)=x_0,$ $\forall k \in K.$\end{itemize}   \end{cor}	
	
	Following the characterization of left symmetric points and the necessary condition for right symmetric points obtained in Theorem \ref{left} and Theorem \ref{right}, respectively, we conclude that there is no non-zero symmetric point in $\ell_{\infty}(K, \mathbb{X}),$ whenever $K$ is not a singleton set. 
	
	\begin{theorem}
		Let $K$ be any non-empty set such that $|K|>1$ and let  $\mathbb{X}$ be a Banach space. Then there is no non-zero symmetric element in $\ell_{\infty}(K, \mathbb{X}).$
	\end{theorem}

The question that naturally arises is whether the conditions mentioned in Theorem \ref{right}	are sufficient. In this connection we have the following result assuming the space $\mathbb{X}$  to be a finite-dimensional real Banach space and the set of all right symmetric points is closed.  To prove the result  we need  the following  characterization of Birkhoff-James orthogonality in $\ell_{\infty}(K, \mathbb{X})$, the proof of which is in the same line of \cite[Th. 2.8]{SPM}.
	
		\begin{theorem}\label{orthogonality:general}
		Let $\mathbb{X}$ be a Banach space and $K$ be  a non-empty set. Let $f,g \in \ell_{\infty}(K, \mathbb{X}).$ Then $g \perp_B f$ if and only if either $(i)$ or $(ii)$ holds: 
		\begin{itemize}
			\item[(i)] There exists a sequence $\{k_n\} \subset K$ such that $\|g(k_n)\| \to \|g\|$ and $f(k_n) \to 0,$ as $n \to \infty.$
			
			\item[(ii)]  there exists two sequence $\{k_n\}, \{t_n\} \subset K$ and $\{\epsilon_n\}, \{\delta_n\} \subset \mathbb{R}$ such that 
			\begin{itemize}
				\item[(a)] $\epsilon_n \to 0, \delta_n \to 0$
				\item[(b)] $\|g(k_n)\| \to \|g\|, \|g(t_n)\| \to \|g\|$
				\item[(c)] $f(k_n) \in g(k_n)^{+\epsilon_n}$ and $f(t_n) \in g(t_n)^{-\delta_n}.$
			\end{itemize}
		\end{itemize}
	\end{theorem}
	%A characterization of $x^+$ and $x^-$ has been obtained in  in terms of supporting functional, which is stated as follows.

	%In \cite{SRBB}, a characterization of right symmetric points has been obtained in real Banach spaces. 
	
	We also need the following lemma.
	
	\begin{lemma}\cite[Th. 2.2]{SRBB}\label{lemma:right}
		Let $\mathbb{X}$ be a real Banach space. Then $x \in S_{\mathbb{X}}$ is  right symmetric if and only if given any $u \in \mathbb{X},$ the following two conditions hold true:
		\begin{itemize}
			\item[(i)] $x \in u^-$ implies that $u \in x^-.$
			\item[(ii)] $x \in u^+$ implies that $u \in x^+.$
		\end{itemize} 
	\end{lemma}

	\begin{theorem}
		Let $K$ be a non-empty set and let $\mathbb{X}$ be a finite-dimensional real Banach space such that the set of all right symmetric points of $S_{\mathbb{X}}$ is closed. Then  $f \in S_{\ell_{\infty}(K, \mathbb{X})}$ is right symmetric if  the following conditions hold: 
		\begin{itemize}
			\item[(i)] $f(k) \in S_{\mathbb{X}},$  $\forall k \in K.$
			\item[(ii)] 	$f(k)$ is right symmetric,  $\forall k \in K.$
		\end{itemize} 
	\end{theorem}
	
	\begin{proof}
		Let  $g \in S_{\ell_{\infty}(K, \mathbb{X})}$ such that $g \perp_B f.$ Since $f(k) \in S_{\mathbb{X}}, \forall k \in K,$ from Theorem \ref{orthogonality:general}, there exist two sequences $\{k_n\}, \{t_n\} \subset K$ and $\{\epsilon_n\}, \{\delta_n\} \subset \mathbb{R}$ such that 
		\begin{itemize}
			\item[(a)] $\epsilon_n \to 0, \delta_n \to 0$
			\item[(b)] $\|g(k_n)\| \to \|g\|, \|g(t_n)\| \to \|g\|$
			\item[(c)] $f(k_n) \in g(k_n)^{+\epsilon_n}$ and $f(t_n) \in g(t_n)^{-\delta_n}.$
		\end{itemize}
		Since $\mathbb{X}$ is finite-dimensional, without loss of generality we assume that $f(k_n) \to z_1, f(t_n) \to z_2, g(k_n) \to w_1, g(t_n ) \to w_2.$ As $f$ satisfies the condition (i), $z_1, z_2 \in S_{\mathbb{X}}.$ As  $f(k_n) \in g(k_n)^{+\epsilon_n}$, we get 
		$$\| g(k_n) + \lambda f(k_n)\| \geq \sqrt{1-\epsilon_n^2} \|g(k_n)\|, \forall \lambda \geq 0. $$
		Taking limit on both sides we get $\|w_1 + \lambda z_1\| \geq \|w_1\|, \forall \lambda \geq 0.$ In other words, $z_1 \in w_1^+.$ Similarly, from $f(t_n) \in g(t_n)^{-\delta_n},$ we obtain $z_2 \in w_2^-.$ Since the set of all right symmetric points of $S_{\mathbb{X}}$ is closed and $f(k_n), f(t_n)$ are right symmetric, $z_1, z_2$ are also right symmetric. Following Lemma \ref{lemma:right} we get, $w_1 \in z_1^+$ and $w_2 \in z_2^-.$
		
		Now, take $y_n^* \in J(f(k_n)).$ As $\mathbb{X}$ is finite-dimensional, as before without loss of generality we assume $y_n^* \to y^* \in S_{\mathbb{X}^*}.$ As $f(k_n) \to z_1,$ it is straightforward to see that $y^*(z_1)=1.$ In other words, $y^* \in J(z_1). $ As $w_1 \in z_1^+,$ it follows easily that
		$y^*(w_1) \geq 0$ (see \cite[Th. 2.4]{SMP}).
		   So, $\lim y_n^*(g(k_n))= y^*(w_1) \geq 0.$ Next, we consider $z_n^* \in J(f(t_n)).$ Following similar arguments and using \cite[Th. 2.4]{SMP} we can show that $\lim z_n^*(g(t_n)) \leq  0. $
		Therefore, \[
		0 \in co(\{ \lim y_n^*(g(k_n)): k_n \in K, y_n^* \in S_{\mathbb{X}^*} \forall n \in \mathbb{N}, \lim y_n^* (f(k_n))=1\}).	
		\]
		Using Theorem \ref{theorem:martin}, $f \perp_B g.$ This proves that $f$ is right symmetric.
	\end{proof}
	
	\begin{remark}
		It is worth mentioning here that the set of  right symmetric points is closed in  finite-dimensional real polyhedral Banach spaces. In fact, we are yet to get an example of  a finite-dimensional real Banach space where the set of right symmetric points is not closed.
	\end{remark}

	We next characterize the right symmetric functions in $C(K, \mathbb{X}),$ where $\mathbb{X}$ is a real Banach space.

	\begin{theorem}\label{right:sufficient}
		Let $K$ be a compact Hausdorff space and let $\mathbb{X}$ be a real Banach space. Then $f \in S_{C(K, \mathbb{X})}$ is right symmetric if and only if $f$ satisfies the following conditions:
		\begin{itemize}
			\item[(i)] $f(k) \in S_{\mathbb{X}},$ $ \forall k \in K.$
			\item[(ii)] 	$f(k)$ is right symmetric, $ \forall k \in K.$
		\end{itemize} 
	\end{theorem}

	\begin{proof}
		We first prove the sufficient part. 
		Let $g \perp_B f.$ From Corollary \ref{ortho:continuous}, there exist $k_1, k_2 \in M_g$ such that $f(k_1) \in g(k_1)^+$ and $f(k_2) \in g(k_2)^-.$ As $f(k_1), f(k_2)$ both are right symmetric points, applying Lemma \ref{lemma:right}, we obtain that $g(k_1) \in f(k_1)^+$ and $g(k_2) \in f(k_2)^-.$ Observe that $k_1, k_2 \in M_f.$ Using Corollary \ref{ortho:continuous}, we get that $f \perp_B g.$ Therefore, $f$ is right symmetric.
		
		Let us now prove the necessary part. First we  show that $f(k) \in S_{\mathbb{X}},$ $ \forall k \in K.$	 Suppose on the contrary that there exists $k_0 \in K$ such that $\|f(k_0)\| <1.$ Take $w_0 \in S_{\mathbb{X}}$ such that $w_0 \perp_B f(k_0).$ From Proposition \ref{compact}, it follows that  $M_f$ is a compact set. Since $K$ is a compact Hausdorff space, there exist open sets $U$ and $V$ such that $U \cap V= \emptyset, k_0 \in U$ and $M_f \subset V.$  Now $\{k_0\} $ and $K \setminus U $ are two disjoint closed sets and so by the Urysohn's lemma \cite{Munkres}, there exists a continuous function $h_1 : K \to [0,1] $ such that $h_1(k_0)=1$ and $h_1(k)=0,$  $\forall k \notin U.$ Similarly, there exists a continuous function $h_2: K \to [0,1]$ such that  $h_2(k)=1,$  $\forall k \in M_f$ and $h_2(k)=0,$  $ \forall k \notin V.$ 
		Define $ g: K \to \mathbb{X}$ as 
		$$ g(k)=h_1(k)w_0+ h_2(k)f(k), ~ \forall k \in K.$$ Clearly, $g $ is continuous. Note that  $g(k)= h_1(k) w_0, \, \forall k \in U, $ $g(k)= h_2(k) f(k), \, \forall k \in V$  and  $g(k)= 0, \, \forall k \in K \setminus (U \cup V),$ which shows that $\|g\|=1.$ Also, $g(k_0)=w_0$ and $g(k)=f(k), \, \forall k \in M_f. $ As $w_0 \perp_B f(k_0),$ we observe that for any scalar $\lambda,$ 
		$$\|g + \lambda f\| \geq \|g(k_0)+ \lambda f(k_0)\|= \|w_0+ \lambda f(k_0)\| \geq \|w_0\|=\|g(k_0)\|=\|g\|.$$  
		So, $g \perp_B f.$ For any $k \in M_f, \, g(k)= f(k)$  and so for any $k \in K,y^* \in Ext(B_{\mathbb{X}^*})$ with  $ y^*(f(k))= \|f\|=1,$ we get  $y^*(g(k))= y^*(f(k))=\|f\|=1.$ So, 
		\[
		0 \notin co(\{y^*(g(k)): k \in K, y^* \in Ext(B_{X^*}), y^*(f(k))=1\}).
		\]
		Therefore, using Corollary \ref{martin:continuous}, we conclude that $f \not\perp_B g.$ This contradicts the fact that $f$ is right symmetric. 
		
		To complete the proof, we only need to show that $f$ satisfies the condition (ii). Suppose on the contrary  that there exists $k_0 \in K$ such that $f(k_0)$ is not right symmetric. Therefore, there exists $w_0 \in S_{\mathbb{X}}$ such that $w_0 \perp_B f(k_0)$ but $f(k_0) \not\perp_B w_0.$ Without loss of generality we assume that $w_0 \in f(k_0)^+.$ So, $w_0 \notin f(k_0)^-.$ Then there exists a scalar $\lambda_0 <0$ such that  $\|f(k_0) + \lambda_0 w_0\| < \|f(k_0)\|.$ Let us now define a function $\zeta: \mathbb{X} \to \mathbb{R}$ such that for any $x \in \mathbb{X},$ $\zeta(x)= \|x+ \lambda_0 w_0\|-\|x\|.$ Clearly, $\zeta$ is continuous and $\zeta(f(k_0))<0.$ Therefore, there exists an open set $V \subset \mathbb{X}$ containing $f(k_0)$ such that for any $v \in V,$ $\zeta(v) <0.$ So, for any $v \in V,$ $\|v+ \lambda_0 w_0\|< \|v\|,$ which implies that $w_0 \notin v^-.$ Thus, $w_0 \in v^+$ and $w_0 \notin v^-,$ for any $v \in V.$ By the continuity of $f$ at $k_0,$ there exists  an open set $U$ of $K$ containing $k_0$ such that $f(U) \subset V.$ As before, using the Urysohn's Lemma, there exists a continuous map $h: K \to [0,1]$ such that $h(k_0)=1$ and $h(K\setminus U)=0.$ Define $g: K \to \mathbb{X}$ as 
		$$ g(k)= (1-h(k)) f(k)+ h(k)w_0, ~\forall k \in K.$$
		Clearly, $g$ is continuous and moreover,  $\|g(k)\|= \|f(k)\|=1, \, \forall k \in K \setminus U$ and   $\|g(k)\| \leq |1-h(k)| \|f(k)\|+|h(k)|\|w_0\| \leq 1, \, \, \forall k \in U.$ So, $\|g\|=1.$ Since $g(k_0)=w_0,$ we observe that for any scalar $\lambda \in \mathbb{R},$
		\[
		\|g+ \lambda f\|\geq \|g(k_0)+\lambda f(k_0)\|=\|w_0+\lambda f(k_0)\|\geq \|w_0\|=1=\|g\|.
		\]
		So, $g \perp_B f.$ For any $k \in K\setminus U$ and $y^* \in J(f(k)),$ $y^*(g(k))= y^*(f(k)) =1.$ For any  $k \in U,$ $g(k) \in f(k)^+$ and $g(k) \notin f(k)^-.$ This implies that for any $k \in U$ and for any $y^*\in Ext(B_{\mathbb{X}^*})$ with $y^*(f(k))=\|f\|,$ we have $y^*(g(k)) >0$ (by Lemma \ref{functional}). Therefore, 
		$$ 0 \notin co(\{y^*(g(k)): k \in K, y^* \in Ext(B_{\mathbb{X}^*}), y^*(f(k))=\|f\|\}).$$ 
		So, using Corollary \ref{martin:continuous}, we get  $f \not\perp_B g.$ This contradicts the fact that $f$ is  right symmetric, and establishes the theorem.
	\end{proof}

		\begin{cor}
			Let $K$ be a compact Hausdorff space and let $\mathbb{X}$ be a real  Banach space such that the set of all right symmetric points in $S_{\mathbb{X}}$ is not connected. Then $K$ is connected if and only if the following are equivalent:
			\begin{itemize}
					\item[(i)] $f$ is right symmetric  in $S_{C(K, \mathbb{X})}$.
						\item[(ii)] there exists a right symmetric point $x_0 \in S_{\mathbb{X}}$ in $\mathbb{X}$ such that $f(k)=x_0,$ $\forall k \in K.$
					\end{itemize}   
				\end{cor}
				
				\begin{proof}
					Let $K$ be connected. If $f$ is right symmetric in $C(K, \mathbb{X})$ then it follows from Theorem \ref{right:sufficient} that for any $k \in K,$ $f(k) \in S_{\mathbb{X}}$ and $f(k)$ is right symmetric in $\mathbb{X}.$ As $K$ is connected and $f$ is continuous, $f(K) \subset S_{\mathbb{X}}$ is also connected. Since  the set of all right symmetric points in $S_{\mathbb{X}}$ is not connected,  we conclude that $f(K)$ is singleton and $f$ satisfies the condition (ii). On the other hand, if $f$ satisfies the condition (ii), then from Theorem \ref{right:sufficient} we infer that $f$ is right symmetric. 
					
					For the converse part, suppose on the contrary that $K$ is not connected. Then there exists two non-empty disjoint open sets $U, V$ such that $K = U \cup V.$ Let $x \in S_{\mathbb{X}}$ be right symmetric. Define $f :K \to \mathbb{X}$   such that $f(U)=x$ and $f(V)=-x.$ Clearly, $f$ is continuous and by virtue of Theorem \ref{right:sufficient}, $f$ is right symmetric. Since $f$ does not satisfy the condition (ii), this completes the proof.                     
				\end{proof}

				\begin{remark}
						Let $K$ be a compact connected  Hausdorff space and let $\mathbb{X}$ be a real Banach space. Suppose that right symmetric points in $S_{\mathbb{X}}$ are finite. Then the number of right symmetric functions in $S_{C(K, \mathbb{X})}$  are same as the number of right symmetric points in $S_{\mathbb{X}}.$ 
				\end{remark}

 As an immediate consequence of Theorem \ref{right:sufficient}, we obtain that the right symmetric points on the unit sphere of $C(K)$ are precisely the extreme points of the unit ball.
	
		\begin{cor}
		Let $K$ be a compact Hausdorff space and let $C(K)$ be the space of all real valued continuous functions on $K.$  Let $f \in S_{C(K)}.$ Then the following are equivalent:
		\begin{itemize}
			\item[(i)] $f$  is a right symmetric function of $C(K).$ 
		\item[(ii)]  $f(k)=1,$ $\forall k \in K$ or $f(k)=-1,$ $\forall k \in K.$
		\item[(iii)]$f$ is an extreme point of $B_{C(K)}.$
		\end{itemize}
	\end{cor}

Next we show that if $K$ is a  locally compact normal space which is not compact, then the  space $C_0(K, \mathbb{X})$ has no non-zero right symmetric points.

\begin{theorem}\label{locally compact}
	Let $K$ be a locally compact normal space which is not compact and let $\mathbb{X}$ be a Banach space. Then $f \in S_{C_0(K, \mathbb{X})}$ is right symmetric if and only if $f$ is the zero function.
\end{theorem}

\begin{proof}
	As the sufficient part follows trivially, we only prove the necessary part. Let $f$ be right symmetric. Suppose on the contrary that $f \neq 0.$ As $f \in C_0(K, \mathbb{X}),$  there exists $k_0 \in K$ such that $\|f(k_0)\|=r \in (0,1).$ Take $w_0 \in S_{\mathbb{X}}$ such that $w_0 \perp_B f(k_0).$ From  Proposition \ref{compact}, it follows that $M_f$ is a compact set. Since $K$ is a locally compact  normal space, there exist open sets $U$ and $V$ such that $U \cap V= \emptyset, k_0 \in U,$ $M_f \subset V$  and $\overline{U}, \overline{V}$  both compact (by Theorem \ref{locally compact hausdorff}).
	As  before, by using the Urysohn's Lemma, we can find two continuous functions $h_1,h_2 : K \to [0,1] $   such that $h_1(k_0)=1$ and $h_1(k)=0,$  $\forall k \notin U,$ and $h_2(k)=1,$  $\forall k \in M_f$ and $h_2(k)=0,$  $ \forall k \notin V.$ 	Define $ g : K \to \mathbb{X}$ as 
	 	$$g(k)=h_1(k)w_0+ h_2(k)f(k), ~ \forall k \in K.$$  
	 	Clearly, $g$ is continuous and for any $k \in K \setminus (\overline{U} \cup \overline{V}),$ $\|g(k)\|=0.$ So, $g \in C_0(K, \mathbb{X}).$ Then proceeding similarly as in the proof of Theorem \ref{right:sufficient}(i), we obtain that $g \perp_B f$ and $f \not \perp_B g.$ This contradicts the fact that $f$ is right symmetric. Thus $f=0.$ 
\end{proof}

As a consequence of Theorem \ref{perfect} and  Theorem \ref{locally compact}, we derive the following results.  
\begin{cor}
	Let $K$ be a locally compact, perfectly normal space such that $K$ is not compact and $K$ has no isolated point.  Then the space $C_0(K, \mathbb{X})$ has no non-zero left symmetric points and no non-zero right symmetric points.
\end{cor}

Considering $K= \{1,2, \ldots, n \} ,$ the following corollary is obvious from  Theorem \ref{left} and Theorem \ref{right:sufficient}.

\begin{cor}\label{directsum}
	Let $\mathbb{X}$ be a Banach space and let $\widetilde{x}=(x_1, x_2, \ldots, x_n) \in S_{\ell_{\infty}^n(\mathbb{X})}.$ Then 
	\begin{itemize}
		\item[(i)] $\widetilde{x}$ is left symmetric if there exists $i_0 \in \{1,2, \ldots, n\}$ such that $x_{i_0} \in S_{\mathbb{X}}$ is left symmetric and $x_j=0,$ for any $j \in \{1, 2, \ldots, n\}\setminus \{i_0\}$.
		
		\item[(ii)] $\widetilde{x}$ is right symmetric if and only if for any $1 \leq i \leq n ,$ $x_i \in S_{\mathbb{X}}$ and $x_i$ is right symmetric, when $\mathbb{X}$ is a real Banach space.
		
		\item[(iii)] $\widetilde{x}$ is symmetric if and only if $\widetilde{x}=0.$ 
	\end{itemize}
\end{cor}
	
%\begin{remark}
%	Let $K$ be a compact connected  Hausdorff space and $\mathbb{X}$ be a real Banach space. Suppose that right symmetric points in $S_{\mathbb{X}}$ of $\mathbb{X}$ are finite. Then there are no non-constant right symmetric functions in $ C(K, \mathbb{X}).$  Moreover  the number of right symmetric functions in $S_{C(K, \mathbb{X})}$  are same as the number of right symmetric points in $S_{\mathbb{X}}.$ 	
%\end{remark}

The symmetric points in a Banach space play a vital role in identifying the onto linear isometries on that space, by virtue of the following fact.

\begin{prop}\cite[Cor. 1.1]{BRS}\label{prop:isometry}
	Let $\mathbb{X}$ and $\mathbb{Y}$ be normed linear spaces and let $T \in \mathbb{L}(\mathbb{X}, \mathbb{Y})$ be an onto linear
	isometry. Then $x \in  \mathbb{X}$ is left symmetric (resp. right symmetric) if and only if $T (x) $ is
	left symmetric (resp. right symmetric) in $\mathbb{Y}$.
\end{prop}

%In this connection the following is immediate.

%begin{prop}
%	Let $\mathbb{X}$ be a Banach space. Let $\mathcal{L}$ and $\mathcal{R}$ be the set of all left symmetric and right symmetric points in $S_{\mathbb{X}},$ respectively. Suppose $\mathcal{L} \cup \mathcal{R}$ separates $\mathbb{X}^*.$
%\end{prop} 

Using this connection between symmetric points and onto isometries,
we provide a necessary condition for onto isometries on the space $\ell_{\infty}^n(\mathbb{X}).$ In the following result for a Banach space $\mathbb{X},$ we denote the set of all norm one left symmetric points as $\mathcal{L}.$ Also for any $ 1\leq i \leq n$ and for any $ x \in \mathbb{X}, $ we write $e_i(x)= (0, 0, \ldots \underset{i-th}{, x,} 0, \ldots, 0) \in \ell_{\infty}^n(\mathbb{X}).$

\begin{theorem}
	Let $\mathbb{X}$ be a  Banach space such that  $span ~ \mathcal{L}= \mathbb{X}.$  Let $T $ be an onto linear isometry on $\ell_{\infty}^n(\mathbb{X})$ . Then there exists a basis $\mathcal{B} \subset \mathcal{L}$ of $\mathbb{X}$ such that for any $(x_1, x_2, \ldots, x_n) \in \ell_{\infty}^n(\mathbb{X}),$
	\[
	T(x_1, x_2, \ldots, x_n)= \sum_{i=1}^{n} \sum_{k=1}^{m_i}\alpha_k^i e_{\sigma(i)} (\psi(z_k)),
	\]
where for any $ 1\leq i \leq n,$ $x_i= \sum_{k=1}^{m_i} \alpha_k^i z_k, $ $z_k \in \mathcal{B}, \forall 1 \leq k \leq m_i,$  $\sigma$ and $\psi $ are  permutations on the set $\{1,2, \ldots, n\}$ and $\mathcal{L},$ respectively. 
 \end{theorem}

\begin{proof}
	Let  $\mathcal{L} \subset S_{\mathbb{X}}$ be the set of all left symmetric points of $S_\mathbb{X}.$ Applying Theorem \ref{left}, the set of left symmetric points of $S_{\ell_{\infty}^n(\mathbb{X})}$ is given by $ \mathcal{S}=\{ e_i(z): 1 \leq i \leq n, z \in \mathcal{L}\}. $  Following Proposition \ref{prop:isometry}, if $T$ is an onto isometry then $T(s) \in \mathcal{S}, \forall s \in \mathcal{S}.$ This implies that for any $1 \leq i \leq n, z \in \mathcal{L},$ 
	$$T(e_i(z))= e_{\sigma(i)}(\psi(z)),$$ 
	where $\sigma$ is a permutation on $\{1, 2, \ldots, n\}$ and $\psi$ is a permutation on $\mathcal{L}.$  For any $(x_1, x_2, \ldots, x_n) \in \ell_{\infty}^n(\mathbb{X}),$ $x_i= \sum_{k=1}^{m_i} \alpha_k z_k,$ where $\alpha_k \in \mathbb{K}, z_k \in \mathcal{L}, \forall 1 \leq k \leq m_i.$
	Suppose $\mathcal{B} \subset \mathcal{L}$ is a basis of $\mathbb{X}$ and $x_i =   \sum_{k=1}^{m_i} \alpha_k^i z_k, $ $z_k \in \mathcal{B}, \forall 1 \leq k \leq m_i.$ Then 
	\begin{eqnarray*}
		T(x_1, x_2, \ldots, x_n)=  \sum_{i=1}^{n} T (e_i(x)) &=& \sum_{i=1}^{n} T (e_i( \sum_{k=1}^{m_i} \alpha_k^i z_k)) \\
		&=& \sum_{i=1}^{n}\sum_{k=1}^{m_i} \alpha_k^i T (e_i(   z_k)) \\
			&=& \sum_{i=1}^{n}\sum_{k=1}^{m_i} \alpha_k^i e_{\sigma(i)}(\psi(z_k)).
	\end{eqnarray*}
\end{proof}

In this context, it is worth noting that it follows from \cite[Th. 2.7, Th 2.9]{CSS} that $span~\mathcal{L}= \mathbb{X},$  whenever $\mathbb{X}= \ell_p^n$ ($1 <p \leq \infty$). 
  It is well-known that the signed permutations are the only onto isometries on $\ell_p^n $ spaces. The concept of symmetric points was utilized in \cite[Th.  2.11]{CSS} to give an elementary proof of the same. Applying similar technique, in \cite[Th. 4.5]{BRS}, a simple proof of the classical Banach-Lamperti Theorem has been given, which states that the signed permutations are the only onto isometries on $\ell_p$ spaces.

	%%%%%%%%%%%%%%%%%%%%%%%%%%%%%%%%%%%%%%%%%%%%%%%%%%%%%%%%%%%%%%%%%%%%%%%%%%
	%%%%%%%%%%%%%%%%%%%%%%%%%%%%%%%%%%%%%%%%%%%%%%%%%%%%%%%%%%%%%%%%%%%%%%%%%%%%
	%%%%%%%%%%%%%%%%%%%%%%%%%%%%%%%%%%%%%%%%%%%%%%%%%%%%%%%%%%%%%%%%%%%%%%%%%%%%%%%55555
	
	\section*{Section-II}
	
	In this section, our primary goal is to study the symmetric points in operator spaces defined on real Banach spaces.   
	For a Banach space $\mathbb{X}$ with $Ext(B_{\mathbb{X}}) \neq \emptyset,$  let $e(B_{\mathbb{X}})= \{ x, y \in Ext(B_{\mathbb{X}}): x \neq \pm y\}.$  We first observe that the operator space $\mathbb{L}(\mathbb{X}, \mathbb{Y})$ can be embedded into the spaces of continuous functions.   
	
	\begin{prop}\label{prop:embedd}
		Let $\mathbb{X}, \mathbb{Y}$ be  Banach spaces. Then $\mathbb{L}(\mathbb{X}, \mathbb{Y})$ is  embedded into $C(e(B_{\mathbb{Y}^*}), \mathbb{X}^*).$ 
	\end{prop}

	\begin{proof}
	%	To prove this we need to show an isometry from $\mathbb{L}(\mathbb{X},\mathbb{Y})$ to $C(e(B_{\mathbb{Y}^*}), \mathbb{X}^*).$ 
We		define $ 	\phi: \mathbb{L}(\mathbb{X},\mathbb{Y}) \quad  \longrightarrow  \quad  C(e(B_{\mathbb{Y}^*}), \mathbb{X}^*) $ as 
	$$ 	\phi(T)\quad  =  \quad f_T, \, \, \forall  T \in \mathbb{L}(\mathbb{X}, \mathbb{Y}) , $$
	where $ f_T:  e(B_{\mathbb{Y}^*}) \to \mathbb{X}^* $ is given by 
			$$ f_T( y^*)=  T^*( y^*), \, \, \forall y^* \in e(B_{\mathbb{Y}^*}).$$
		As $T$ is continuous and $f_T= T^*|_{e(B_{\mathbb{Y}^*})},$  $f_T$ is continuous. Clearly, $\phi$ is well-defined. We show that $\phi$ is a linear isometry.
First we observe that for any $k \in K$ and for any scalar $\alpha, \beta,$
		\begin{align*}
			f_{\alpha T_1+ \beta T_2}(y^*)= (\alpha T_1 + \beta T_2)^*(y^*)= \alpha T_1^*(y^*)+ \beta T_2^*(y^*)= (\alpha f_{T_1}+ \beta f_{T_2}) (y^*).
		\end{align*}
		This implies that $\phi( \alpha T_1+ \beta T_2)= \alpha \phi(T_1) + \beta \phi(T_2).$ 
		Also,
		\begin{align*}
			\|\phi(T)\|=\|f_T\|= \sup_{y^* \in e(B_{\mathbb{Y}^*})} \|f_T(y^*)\| =& \quad  \sup_{y^* \in e(B_{\mathbb{Y}^*})} \|T^*(y^*) \|\\ =& \quad  \sup_{y^* \in e(B_{\mathbb{Y}^*})} \sup_{x \in S_{\mathbb{X}}} |T^*(y^*) (x)| \\
			=& \quad  \sup_{x \in S_{\mathbb{X}}} \sup_{y^* \in e(B_{\mathbb{Y}^*})} |T^*(y^*) (x)| \\
			=& \quad  \sup_{x \in S_{\mathbb{X}}} \sup_{y^* \in e(B_{\mathbb{Y}^*})} |y^* (Tx)|.
	%		=& \quad \sup_{x \in S_{\mathbb{X}}} \|Tx\|  \\
	%		=&  \quad \|T\|.
		\end{align*} 
	Following \cite[Cor. 2.10.7]{Megg}, $\sup_{y^* \in e(B_{\mathbb{Y}^*})}|y^*(Tx)|= \|Tx\|,$ which implies that $$\|\phi(T)\|=\sup_{x \in S_{\mathbb{X}}} \|Tx\|= \|T\|. $$	
		This proves that $\phi$ is an isometry from $\mathbb{L}(\mathbb{X},\mathbb{Y})$ into $C(e(B_{\mathbb{Y}^*}), \mathbb{X}^*),$ finishing the proof.
		
	\end{proof}

    Next we present    sufficient conditions for an operator to be left symmetric and right symmetric. 
    %The proof is immediate as it follows directly by    using Proposition \ref{prop:embedd} and  the sufficient conditions for left symmetric and right symmetric points  obtained in Theorem \ref{left} and \ref{right:sufficient}, respectively.

      \begin{theorem}\label{nice:left}
    	Let $\mathbb{X}, \mathbb{Y}$ be Banach spaces and  let $T \in S_{\mathbb{L}(\mathbb{X}, \mathbb{Y})}.$ Then 
    	\begin{itemize}
    		\item[(i)]  $T$ is left symmetric if  there exists $y_0^* \in Ext(B_{\mathbb{Y}^*})$ such that $T^*(y_0^*) $ is a left symmetric point and $T^*(y^*)=0, \forall y^* \in Ext(B_{\mathbb{Y}^*}) \setminus \{\pm y_0^*\}.$
    		\item [(ii)] $T$ is right symmetric if for each $y^* \in Ext(B_{\mathbb{Y}^*}),$ $T^*(y^*) $ is a right symmetric point of $S_{\mathbb{X}^*}$ and 	$Ext(B_{\mathbb{Y}^*})$ is a compact set in $\mathbb{Y}^*.$
    	\end{itemize} 
    \end{theorem}

    \begin{proof}
   (i)  We define $ 	\phi: \mathbb{L}(\mathbb{X},\mathbb{Y}) \quad  \longrightarrow  \quad  \ell_{\infty}(Ext(B_{\mathbb{Y}^*}), \mathbb{X}^*) $  as  
    	  $$ 	\phi(T)\quad  =  \quad f_T \quad   \forall  T \in \mathbb{L}(\mathbb{X}, \mathbb{Y}) , $$
    	  where $ f_T:  Ext(B_{\mathbb{Y}^*}) \to \mathbb{X}^* $ is given by 
    	  $$ f_T( y^*)=  T^*( y^*), \, \, \forall y^* \in Ext(B_{\mathbb{Y}^*}).$$
    	   Proceeding similarly as in the proof of Proposition \ref{prop:embedd}, we can show that $\phi$ is a linear isometry. So, $\mathbb{L}(\mathbb{X}, \mathbb{Y})$ is embedded into $\ell_{\infty}(Ext(B_{\mathbb{Y}^*}), \mathbb{X}^*).$
    As 	there exists $y_0^* \in Ext(B_{\mathbb{Y}^*})$ such that $f_T(y_0^*) $ is a left symmetric point and $f_T(y^*)=0, \forall y^* \in Ext(B_{\mathbb{Y}^*}) \setminus \{\pm y_0^*\}, $ it follows from Theorem \ref{left} that $ f_T$ is left symmetric in $\ell_{\infty}(Ext(B_{\mathbb{Y}^*}), \mathbb{X}^*).$ Clearly, $f_T$ is also  left symmetric  in $\phi(\mathbb{L}(\mathbb{X}, \mathbb{Y})).$ Since  left symmetricity is preserved under isometric isomorphism, $T$ 
     is left symmetric in $\mathbb{L}(\mathbb{X}, \mathbb{Y}).$	
     \smallskip
     
     (ii) Consider   $$ 	\phi: \mathbb{L}(\mathbb{X},\mathbb{Y}) \quad  \longrightarrow  \quad  C(Ext(B_{\mathbb{Y}^*}), \mathbb{X}^*) ,$$ defined as 
     $$ 	\phi(T)\quad  =  \quad f_T \quad   \forall  T \in \mathbb{L}(\mathbb{X}, \mathbb{Y}) , $$
     where $ f_T:  Ext(B_{\mathbb{Y}^*}) \to \mathbb{X}^* $ is given by 
     $$ f_T( y^*)=  T^*( y^*), \, \, \forall y^* \in Ext(B_{\mathbb{Y}^*}).$$
     It is easy to see that  $\phi$  is a linear isometry from $\mathbb{L}(\mathbb{X},\mathbb{Y})$ into $ C(Ext(B_{\mathbb{Y}^*}), \mathbb{X}^*) $. Since   $f_T(y^*) \in S_{\mathbb{X}^*}$ is right symmetric for each  $y^* \in Ext(B_{\mathbb{Y}^*}),$ it follows from  Theorem \ref{right:sufficient} that  $f_T$ is right symmetric in $C(Ext(B_{\mathbb{Y}^*}), \mathbb{X}^*).$  Clearly, $f_T$ is   right symmetric  in $\phi(\mathbb{L}(\mathbb{X}, \mathbb{Y})).$ Since   right symmetricity is preserved under isometric isomorphism,
     $T$ is right symmetric in $\mathbb{L}(\mathbb{X}, \mathbb{Y}).$
    	
 % Observe that   using  Proposition \ref{prop:embedd} 
  %and the fact that $C(e(B_{\mathbb{Y}^*}), \mathbb{X}^*)$ is embedded in  $\ell_{\infty}(Ext(B_{\mathbb{Y}^*}), \mathbb{X}^*)$,
  % it is clear that $\mathbb{L}(\mathbb{X}, \mathbb{Y})$ is embedded in  $\ell_{\infty}(Ext(B_{\mathbb{Y}^*}), \mathbb{X}^*)$ and also note that $ 	\phi: \mathbb{L}(\mathbb{X},\mathbb{Y}) \quad  \longrightarrow  \quad  C(e(B_{\mathbb{Y}^*}), \mathbb{X}^*) $ defined as  

   % Moreover the isometry $\phi$ is defined as follows:  for $T \in \mathbb{L}(\mathbb{X}, \mathbb{Y}),$ $\phi(T)(y^*)= T^*(y^*), \forall y^* \in Ext(B_{\mathbb{Y}^*}).$ As $T^*(y_0^*) $ is a left symmetric point and  $ \forall y^* \in Ext(B_{\mathbb{Y}^*}) \setminus \{\pm y_0^*\},$ $T^*(y^*)=0$, following Theorem \ref{left},    $T$ is left symmetric.
    \end{proof}

%\begin{remark}
%Several examples of left symmetric operators and right symmetric operator have been obtained \cite{KRS,PMW}, and all those examples of left and right symmetric operators fall in the category of the above given sufficient conditions of Proposition \ref{nice:left} and \ref{nice:right}, respectively. This leads us to raise the following open questions: for any Banach  spaces $\mathbb{X}, \mathbb{Y},$
%	\begin{itemize}
%	\item	[(i)] Does there exist a left symmetric operator in $\mathbb{L}(\mathbb{X}, \mathbb{Y})$ which does not satisfy the sufficient condition given in Proposition \ref{nice:left}?
	
%	\item[	(ii)] Does there exist a right symmetric operator in $\mathbb{L}(\mathbb{X}, \mathbb{Y})$ which does not satisfy the sufficient condition given in Proposition \ref{nice:right}?
%\end{itemize}

%\end{remark}

	For some special Banach spaces, the Proposition  \ref{prop:embedd} can be strengthened as follows. 
	
	\begin{prop}\label{prop:isometric}
	%	Let $\mathbb{X}$ be a  Banach space. Then the following hold:
		\begin{itemize}
			\item[(i)] The space  $\mathbb{L}(\ell_{1}^n, \mathbb{Y})$ is isometrically isomorphic to $\ell_{\infty}^n(\mathbb{Y}),$ for any Banach space $\mathbb{Y}.$ 
			
				\item[(ii)]  The space  $\mathbb{K}( \mathbb{X}, C(K))$ is isometrically isomorphic to $C(K, \mathbb{X}^*),$ where $K$ is a compact Hausdorff space.
				
			\item[(iii)]  The space $\mathbb{L}( \mathbb{X}, \ell_{\infty}^n)$ is isometrically isomorphic to $\ell_{\infty}^n(\mathbb{X}^*),$
 for any Banach space $\mathbb{X}.$ 			
			\end{itemize}
	\end{prop}
	
	\begin{proof}
		(i) Define $\phi: \mathbb{L}(\ell_{1}^n, \mathbb{Y}) \to \ell_{\infty}^n(\mathbb{Y})$ as  
		$$\phi(T)= (Te_1, Te_2, \ldots, Te_n),$$
		 where $\{e_1, e_2, \ldots, e_n\}$ is the standard ordered basis of $\ell_{1}^n.$ It is clear that $\phi $ is linear. Observe that for any $T \in \mathbb{L}(\ell_1^n, \mathbb{X}),$ $$\|\phi (T)\|= \|(Te_1, Te_2, \ldots, Te_n)\|= \sup \{\|Te_i\|: i \in \{1, 2, \ldots, n\}\}.$$ Since $Ext(B_{\ell_{1}^n})= \{ \pm e_1,\pm  e_2, \ldots, \pm e_n\},$  it is straightforward to observe that
			 $$ \|T\|= \sup \{\|Tx\|: x \in S_{\ell_{1}^n}\}=\sup \{\|Te_i\|: i \in \{1, 2, \ldots, n\}\}.$$ 
		This implies $\|\phi(T)\|= \|T\|.$ So, $\phi$ is an isometry. So to prove (i), we only need to show that $\phi $ is surjective. Let $(x_1, x_2, \ldots, x_n) \in \ell_{\infty}^n (\mathbb{Y}).$ We finish the proof by defining a linear operator $T: \ell_{1}^n \to \mathbb{Y}$ given by $T(e_i)= x_i,$ for any $i, 1 \leq i \leq n.$ \\

			(ii) Follows directly from \cite[Th.1 (p. 490)]{DS}. Indeed, the isometric  isomorphism  $\phi$ between $\mathbb{K}(\mathbb{X}, C(K))$ and $C(K, \mathbb{X}^*)$ is  defined as follows : 
			for  $T \in \mathbb{K}(\mathbb{X}, C(K)),$ $\phi(T)(k)(x)=Tx(k)= T^*(\delta_{k})(x),  \forall k \in K, x \in \mathbb{X},$ where $\delta_k$ is the evaluation map defined on $C(K)$ as $\delta_{k}(f)= f(k), \forall f \in C(K).$ \\

		(iii) This follows immediately from (ii) by taking $|K|=n.$\\

	\end{proof}

In the following theorem, we present a complete characterization of left and right symmetric operators in the space $\mathbb{L}(\ell_{1}^n, \mathbb{X}).$
	
	\begin{theorem}
		Let $\mathbb{X}$ be a Banach space and let $T \in \mathbb{L}( \ell_{1}^n, \mathbb{X}).$  Then
		
			\begin{itemize}
			
			\item[(i)] $T \in S_{\mathbb{L}( \ell_{1}^n, \mathbb{X})}$ is left symmetric if and only if there exists $i_0 \in \{1, 2, \ldots, n\} $ such that the following hold:
			\subitem(a)	  $T(e_{i_0}) \in S_{\mathbb{X}},$ 
				\subitem(b)  $T(e_j)=0,$  $\forall j \in \{1, 2, \ldots, n\} \setminus \{i_0	\},$
			\subitem(c)  $T(e_{i_0})$ is a left symmetric point,
		
			\item[(ii)] $T \in S_{\mathbb{L}( \ell_{1}^n, \mathbb{X})}$ is right symmetric if only if for each $i, 1 \leq i \leq n,$ $T(e_i) \in S_{{\mathbb{X}}}$ and $Te_i$  is  right symmetric,

			\item[(iii)] $T$ is symmetric if and only if $T$ is the zero operator,
		\end{itemize}
		where for $1 \leq i \leq n, $ $e_i= (0,0, \ldots, \underset{i\text{-th}}{1}, 0, \ldots, 0) \in \ell_{1}^n.$
	\end{theorem}

\begin{proof}
	From Proposition \ref{prop:isometric}(i), $\mathbb{L}(\ell_{1}^n, \mathbb{X})$ is isometrically isomorphic to $\ell_{\infty}^n(\mathbb{X}).$ Moreover, $\phi: \mathbb{L}(\ell_{1}^n, \mathbb{X}) \to \ell_{\infty}^n(\mathbb{X})$ defined as  $\phi(T)= (Te_1, Te_2, \ldots, Te_n)$ is the isometric isomorphism. Since  left (right) symmetricity is preserved under isometric isomorphism, the desired result follows easily from Corollary \ref{directsum}.
\end{proof}

\begin{remark}
		In \cite{KRS}, the symmetric operators have been studied in the space $\mathbb{L}(\ell_{1}^n, \mathbb{X})$ and a complete characterization of left symmetric operators has been obtained  \cite[Th 3.6]{KRS}, whereas for the right symmetric operators, only a necessary condition has been found  \cite[Th 3.7]{KRS}.	
	It should be  noted that for the necessary condition obtained in \cite[Th. 3.7]{KRS} for right symmetric operators, we observe that it is also sufficient.
\end{remark}

In the following two results, we characterize the left symmetric operators and the right symmetric operators in the space $\mathbb{K}( \mathbb{X}, C(K)).$

\begin{prop}\label{left1}
	Let $K$ be a compact, perfectly normal space and let $\mathbb{X}$ be a Banach space. Let  $T \in S_{\mathbb{K}(\mathbb{X}, C(K))}.$ Then   $T$ is left symmetric if and only if $T$ satisfies the following: 
	\begin{itemize}
		\item[(i)] there exists exactly one point $k_0 \in K$ such that $T^*(\delta_{k_0}) \in S_{\mathbb{X}}$ and $T^*(\delta_{k})=0,$  $\forall k \in K \setminus \{k_0\},$
		\item[(ii)] $T^*(\delta_{k_0})$ is a left symmetric point in $\mathbb{X}^*,$
	\end{itemize} 
	where $\delta_k \in (C(K))^*$ such that  $\delta_k(f)=f(k),$ for any $f \in C(K).$
	Moreover, if $K$ does not contain any isolated point then $T$ is left symmetric if and only $T$ is the zero operator. 
\end{prop}

\begin{proof}
		From Proposition \ref{prop:isometric}(ii), $\mathbb{K}( \mathbb{X}, C(K))$ is isometrically isomorphic to $ C(K, \mathbb{X}^*).$ Moreover, $\phi: \mathbb{K}( \mathbb{X}, C(K)) \to C(K, \mathbb{X}^*)$ given by  $\phi(T)(k)= T^*(\delta_{k})$ is the isometric isomorphism.  Hence applying Theorem \ref{perfect}, the result follows easily. 
\end{proof}

\begin{prop}
	Let $K$ be a compact Hausdorff space and $\mathbb{X}$ be a Banach space. Let $T \in S_{\mathbb{K}(\mathbb{X}, C(K))}.$ Then $T$ is a right symmetric operator if and only if the following hold: 
	\begin{itemize}
		\item[(i)]  $T^*(\delta_k) \in S_{\mathbb{X}^*},$ for any $k \in K,$
		\item[(ii)] $T^*(\delta_{k})$ is right symmetric point of $S_{\mathbb{X}^*},$ $ \forall k \in K,$
	\end{itemize} 
	where $\delta_k \in (C(K))^*$ such that  $\delta_k(f)=f(k),$ for any $f \in C(K).$
\end{prop}
\begin{proof}
The proof is in the same line as that of Proposition \ref{left1}, and is therefore omitted. 
\end{proof}

	Next we provide an explicit form of  the left and right symmetric operators in $\mathbb{L}(\mathbb{X}, \ell_{\infty}^n).$

	\begin{theorem}\label{left:infty}
		Let $\mathbb{X}$ be any Banach space and let $T \in S_{\mathbb{L}(\mathbb{X}, \ell_{\infty}^n)}.$ Then
		\begin{itemize}
			\item[(i)] $T$ is left symmetric if and only if there exists $j \in \{1, 2, \ldots, n\}$  such that for any $x \in \mathbb{X},$ $Tx=(0, \ldots, 0, \underset{j\text{-th}}{f(x)}, 0, \ldots, 0),$  where $f \in S_{\mathbb{X}^*}$ is a  left symmetric point of $ \mathbb{X}^*$.  
				\item[(ii)] $T$ is right symmetric if and only if  for any $x \in \mathbb{X},$ $T(x) =(f_1(x), f_2(x), \ldots, f_n(x)),$ where  each $f_i \in S_{\mathbb{X}^*}$ is a right symmetric point of $\mathbb{X}^*,$ for any $1\leq i \leq n.$
				\item[(iii)] $T$ is symmetric if and only if $T$ is the zero operator.
		\end{itemize}  
	\end{theorem}
	
	\begin{proof}
	Observe that $Ext(B_{(\ell_{\infty}^n)^*})= \{\pm e_i:  1 \leq i \leq n\},$ where $e_i(x)=x_i,$ for any $x=(x_1, x_2, \ldots, x_n)\in \ell_{\infty}^n.$	From Proposition \ref{prop:isometric}, $\mathbb{L}(\mathbb{X}, \ell_{\infty}^n)$ is isometrically isomorphic to $\ell_{\infty}^n(\mathbb{X}^*).$ Moreover, if $\phi$ is the isometric isomorphism between $\mathbb{L}(\mathbb{X}, \ell_{\infty}^n)$ and    $\ell_{\infty}^n(\mathbb{X}^*),$ then $\phi(T)=(T^*e_1, T^* e_2, \ldots, T^* e_n).$ 
		
		(i) From Corollary \ref{directsum}, we obtain that $\phi(T)$ is left symmetric if and only if there exists $j \in \{1, 2, \ldots, n\} $ such that $T^*(e_{j}) \in S_{\mathbb{X}^*}$ is a left symmetric point and $T^*(e_i)=0,$ for any $i\in \{1,2,\ldots,n\} \setminus \{ j\}.$  Suppose that  $T^*(e_{j})=f.$
		Let for any $x \in \mathbb{X},$ $Tx= (u_1(x), u_2(x), \ldots, u_n(x)),$ where $u_i \in \mathbb{X}^*.$ Then $u_{j}(x)=e_{j}(Tx)= T^*(e_{j})(x)=f(x)$ and $u_i(x)=e_i(Tx)= T^*(e_i)(x)= 0,$ whenever $i \neq j.$ Thus we obtain (i).
		
		(ii) Proceeding similarly as in (i)  and applying Corollary \ref{directsum}, we obtain (ii). 
		
		(iii) Follows directly from (i) and (ii).
	\end{proof}

\begin{remark}
	It is clear from the previous result that $\mathbb{L}(\mathbb{X}, \ell_{\infty}^n)$ has no non-zero left symmetric point  if $\mathbb{X}^*$ does not have any non-zero left symmetric point. As $(\ell_{\infty}^n)^* = \ell_1^n$ and $\ell_{1}^n$ has no non-zero left symmetric point for $ n >2$   (see   \cite[Th. 2.8]{CSS}),
	  there does not exist any non-zero left symmetric operator  in $\mathbb{L}(\ell_{\infty}^n),$ whenever $n >2.$
\end{remark}	

  From Theorem \ref{left:infty} and by using the characterization of left symmetric points of $\ell_p^n, p \neq 1,2, \infty$ (see  \cite[Th. 2.7]{CSS}), we classify    left symmetric operators in $\mathbb{L}(\ell_p^n, \ell_{\infty}^n)$ in the following result.
	
	\begin{cor}\label{left:l_p}
		Let $T \in \mathbb{L}(\ell_p^m, \ell_{\infty}^n)$ with $\|T\|=1.$ Then $T$ is left symmetric if and only if there exists $j \in \{1, 2, \ldots, n\}$ such that $T(x_1, x_2, \ldots, x_m)=(0, 0, \ldots, u_{j}, 0, \ldots, 0),$ where $u_{j}$ is one of the two following:
		\begin{itemize}
			\item[(i)] $u_{j}= x_k,$ for some $k \in \{1, 2, \ldots, m\}.$
			\item[(ii)] $u_{j}= \pm \frac{1}{2^{\frac{1}{q}}} x_{k} \pm \frac{1}{2^{\frac{1}{q}}} x_{l}, $ for some $k,l \in \{1, 2, \ldots, m\}.$ 
		\end{itemize}
	\end{cor}
	
	\begin{proof}
		From Theorem \ref{left:infty}, $T$ is left symmetric if and only if there exists $j \in \{1, 2, \ldots, n\}$ such that for any $x \in \ell_p^m,$ $Tx=(0,0, \ldots, f(x), 0, \ldots, 0),$ where $f \in S_{(\ell_p^m)^*}$ is a left symmetric point of $(\ell_p^m)^*.$ As $(\ell_p^m)^*$ is isometrically isomorphic to $\ell_q^m,$  $f \in S_{(\ell_p^m)^*}$ is exactly of one of the following forms \cite[Th. 2.7]{CSS}:
		\begin{itemize}
			\item[(i)] $f(x_1, x_2, \ldots, x_n)= x_j,$ for some $j, 1\leq j \leq m $
			\item[(ii)] $f(x_1, x_2, \ldots, x_n)= \pm \frac{1}{2^{\frac{1}{q}}} x_{j} \pm \frac{1}{2^{\frac{1}{q}}} x_{k}, $ for some $j,k \in \{1, 2, \ldots, m\}.$ 
		\end{itemize}
		
	\end{proof}

Using Theorem \ref{left:infty} and the characterization of right symmetric points of $\ell_p^n$, $p \neq 1,2, \infty$ (see  \cite[Th. 2.7]{CSS}), the right symmetric operators in the space $\mathbb{L}(\ell_p^m, \ell_{\infty}^n)$ can be expressed explicitly in the following way.
The proof is omitted as it is in the same line as that of Corollary \ref{left:l_p}. 
\begin{cor}
		Let $T \in \mathbb{L}(\ell_p^m, \ell_{\infty}^n)$ with $\|T\|=1.$ Then $T$ is right symmetric if and only if  $T(x_1, x_2, \ldots, x_m)=(u_1, u_2, \ldots, u_n),$ where each $u_{j}$ is one of the following:
	\begin{itemize}
		\item[(i)] $u_{j}= x_k,$ for some $k \in \{1, 2, \ldots, m\}. $
		\item[(ii)] $u_{j}= \pm \frac{1}{2^{\frac{1}{q}}} x_{k} \pm \frac{1}{2^{\frac{1}{q}}} x_{l}, $ for some $k,l \in \{1, 2, \ldots, m\}.$ 
	\end{itemize}
\end{cor}

	\begin{remark}
		The symmetricity of operators on a strictly convex Banach space space has been studied in \cite{KRS, PMW}. In \cite{PMW}, it has been shown that if $\mathbb{X},\mathbb{Y}$ are both reflexive strictly convex spaces then the zero operator is the only left symmetric operator.
		In \cite[Th. 3.3]{KRS}, an example of non-zero left symmetric operator has been given in $\mathbb{L}(\ell_p^2, \ell_{\infty}^2).$  To the  best of our knowledge, it is the only example of a non-zero left symmetric operator on a strictly convex Banach space. However, from Theorem \ref{left:infty}, we now have  a class of non-zero left symmetric operators on strictly convex Banach spaces,  see Corollary \ref{left:l_p}. 
	\end{remark}

We next present  necessary conditions for a rank $1$ operator to be left symmetric and  right symmetric in the space of all compact operators on reflexive Banach spaces. To do so, we need the following easy observation.

\begin{prop}\label{rank:prop}
		Let $\mathbb{X}, \mathbb{Y}$ be  Banach spaces. Let $T \in \mathbb{L}(\mathbb{X}, \mathbb{Y})$ be a $rank~1$ operator. Then $M_T=\pm F,$ for some face $F$ of $B_{\mathbb{X}}.$
\end{prop}

\begin{theorem}\label{left:rank}
	Let $\mathbb{X}, \mathbb{Y}$ be two  reflexive Banach spaces and  let $T \in \mathbb{L}(\mathbb{X}, \mathbb{Y})$ be a $rank~1$   operator. 
	\begin{itemize}
		\item [(i)] If $T$ is left symmetric then  there exists a left symmetric element $w$ in $\mathbb{Y}$ and a left symmetric functional $f$  in $\mathbb{X}^*$ such that 
		\[ Tx = f(x)w, \,\forall x \in \mathbb X.\]  
		\item [(ii)]  If $T$ is right symmetric then there exists a right symmetric element $w$ in $\mathbb{Y}$ and a right symmetric functional $f$  in $\mathbb{X}^*$ such that 
		\[ Tx = f(x)w, \,\forall x \in \mathbb X.\]  
	\end{itemize}
\end{theorem}

	\begin{proof}
	 (i) As $T$ is of rank $1,$ $T$ can be expressed as  $T(x)= f(x)w,$  $ \forall x\in \mathbb{X},$    where $w \in \mathbb{Y}$ and $f \in \mathbb{X}^*.$ Suppose on the contrary that $w$ is not left symmetric. Then there exists $v \in \mathbb{Y}$ such that $w \perp_B v$ but $v \not\perp_B w.$ Define $S : \mathbb{X} \to \mathbb{Y}$ such that $S(x)= f(x) v,$ for any $x \in \mathbb{X}.$ It is straightforward to verify that $M_T= M_S$ and for any $z \in M_T(= M_S),$ $Tz= \pm w, Sz= \pm v.$ Observe that for any $z \in M_T,$ $Tz \perp_B Sz,$ which implies $T \perp_B S.$ Moreover, for any $z \in M_S,$ $Sz \not \perp_B Tz.$ Using Proposition \ref{rank:prop} and \cite[Th.2.1]{SP}, we conclude that $S \not \perp_B T,$ which contradicts that $T$ is left symmetric. Therefore, $w$ is left symmetric. It is easy to check that $T^*(y^*)=  \psi(w)(y^*) f,$ for any $y^* \in \mathbb{Y}^*,$ where $\psi$ is the canonical isometric isomorphism from $\mathbb{X}$ to $\mathbb{X}^{**}.$ Clearly, $T^*$ is rank $1$. Moreover, as $\mathbb{X}, \mathbb{Y}$ are reflexive,  $T$ is left symmetric implies that $T^*$ is also left symmetric. Now, suppose on the contrary that $f$ is not left symmetric. Then there exists $g \in \mathbb{X}^*$ such that $f \perp_B g$ but $g \not \perp_B f.$ Define $A: \mathbb{Y}^* \to \mathbb{X}^*$ such that $A(y^*)= \psi(w)(y^*) g.$ Clearly, $M_{T^*}= M_A.$ Proceeding similarly as above, we obtain that $T^* \perp_B A$ but $A \not\perp_B T^*,$ which contradicts that $T^*$ is left symmetric. This completes the proof of (i). 
	 
	 \smallskip
	 
	 (ii) As in (i) $T$ can be expressed as  $T(x)= f(x)w,$ $ \forall x\in \mathbb{X},$  Let $M_T= \pm F$, for some face $F$ of $B_{\mathbb{X}}$ and $Tx=w,$ for any $x \in F.$ On the contrary, assume that $w$ is not right symmetric. Then there exists $v \in \mathbb{Y}$ such that $v \perp_B w$ but $w \not \perp_B v.$ Define $S: \mathbb{X} \to \mathbb{Y}$ such that $S(x)= f(x) v,$ for any $x \in \mathbb{X}.$ 
	 It is straightforward to see that $M_T= M_S$ and for any $z \in M_T(= M_S),$ $Tz= \pm w, Sz= \pm v.$ Observe that for any $z \in M_S,$ $Sz \perp_B Tz,$ which implies that $S \perp_B T.$ Moreover, for any $z \in M_T,$ $Tz \not \perp_B Sz.$ Using Proposition \ref{rank:prop} and \cite[Th. 2.1]{SP}, we obtain that $T \not \perp_B S,$ which contradicts that $T$ is right symmetric. Therefore, $w$ is right symmetric. Also, we have that $T^*(y^*)=  \psi(w)(y^*) f,$ for any $y^* \in \mathbb{Y}^*,$ where $\psi$ is the canonical isometric isomorphism from $\mathbb{X}$ to $\mathbb{X}^{**}.$ Clearly, $T^*$ is of rank $1$. Moreover, since  $\mathbb{X}, \mathbb{Y}$ are reflexive and  $T$ is right symmetric,  it follows  that $T^*$ is right symmetric. Proceeding similarly as above (the proof of right symmetricity of $w$), we can show that $f$ is right symmetric. This completes the theorem. 
	 	\end{proof}

In \cite[Th. 2.8]{PMW}, it has been shown that for a reflexive and strictly convex Banach space $\mathbb{X},$ if a non-zero $T \in \mathbb{K}(\mathbb{X}, \mathbb{Y})$ is left symmetric, then $T$ is a $rank~1$ operator.  Also in the same article \cite[Th. 2.10]{PMW}, it also has been shown that if both $\mathbb X, \mathbb Y$ are reflexive spaces and  $\mathbb{Y}$ is smooth, then $\mathbb{K}(\mathbb{X}, \mathbb{Y})$ has only $rank~1$ non-zero left symmetric operator. Thus combining Theorem \ref{left:rank} and \cite[Th.2.8, Th.2.10]{PMW}, we can conclude that the left symmetric compact operators   are of special forms as mentioned below.

\begin{theorem}
	Let $\mathbb{X}, \mathbb Y $ be reflexive  Banach spaces and let $T \in S_{\mathbb{K}(\mathbb{X}, \mathbb{Y})}$ be left symmetric.  If either $\mathbb X$ is strictly convex  or $\mathbb Y$ is smooth,  
	then there exists a left symmetric element $w$ in $\mathbb{Y}$ and a left symmetric functional $f$  in $\mathbb{X}^*$ such that 
	\[ Tx = f(x)w, \,\forall x \in \mathbb X.\]  
\end{theorem}

Observe that the  conditions in Theorem \ref{left:rank}  are not sufficient. Indeed, for a Hilbert space $\mathbb{H},$ the space $\mathbb{L}(\mathbb{H})$ has no non-zero left symmetric operators  \cite[Th.3.3]{T}. Moreover, $\mathbb{L}(\mathbb{H})$ does not have any $rank ~1$ right symmetric operator as the only right symmetric operators in $\mathbb{L}(\mathbb{H})$ are the isometries and the co-isometries \cite[Cor.4.5]{T}.


\begin{thebibliography}{99}
		
		
		\bibitem{B} Birkhoff, G.,  \textit{Orthogonality in linear metric spaces}, Duke Math. J. \textbf{1} (1935) 169--172.
		
	%	\bibitem{BS} Bhatia, R. and  $\check{S}$emrl, P., \textit{Orthogonality of matrices and distance problem},  Linear Algebra Appl., \textbf{287} (1999)  77--85.
		
		
		\bibitem{BRS} Bose, B., Roy, S., Sain, D., \textit{Birkhoff-James orthogonality and its local symmetry in some sequence spaces}, RACSAM, \textbf{117} (2023), https://doi.org/10.1007/s13398-023-01420-y.
		
		
		
		\bibitem{CSS} Chattopadhyay, A.,  Sain, D., Senapati, T., \textit{Characterization of symmetric points in $\ell_p^n$ spaces}, Linear Multilinear Algebra, \textbf{69} (2021) 2998-3009.
		
		
		\bibitem{DS}  Dunford, N.,  Schwartz, J, \textit{Linear Operators, Vol. I}, Interscience, New York, 1958.
		
		\bibitem{J}James, R. C., \textit{Orthogonality and linear functionals in normed linear spaces}, Trans.  Amer. Math. Soc., \textbf{61} (1947) 265-292.
		
		
		\bibitem{GSP} Ghosh, P., Sain, D., Paul, K., \textit{On symmetry of Birkhoff-James orthogonality of linear operators}, Adv. Oper. Theory, \textbf{2} (2017) 428-434.
		
		
			\bibitem{KRS} Khurana, D., Roy, S., Sain, D., \textit{Symmetric points in spaces of linear operators	between Banach spaces}, Acta Sci. Math.(Szeged), \textbf{86} (2020) 617-634.
			
			
			\bibitem{Munkres} Munkres, J. R., \textit{Topology}, 2nd edition, Pearson education limited, ISBN 978-93-325-4953-1.
		
		\bibitem{MMQRS} Mart\'{i}n, M., Mer\'{i}, J., Quero, A., Roy, S., Sain, D., \textit{A numerical range approach to Birkhoff-James orthogonality with applications} https://doi.org/10.48550/arXiv.2306.02638.
		
		\bibitem{Megg} Megginson, R. E., \textit{An introduction to Banach space theory}, Graduate Texts in Mathematics, 183 (1998) Springer-Verlag, New York.
		
		
		
		\bibitem{MPS} Mal, A., Paul, K., Sain, D., \textit{Birkhoff–James
			 Orthogonality and Geometry of Operator Spaces}, Infosys Science Foundation Series, Springer Singapore, 2024. ISBN 978-981-99-7110-7, https://doi.org/10.1007/978-981-99-7111-4.
		 
		 
		
		
		\bibitem{PMW} Paul, K., Mal, A., W\'{o}jcik, \textit{Symmetry of Birkhoff-James orthogonality of operators defined between infinite dimensional Banach spaces}, Linear Algebra Appl., \textbf{563} (2019) 143-153.
		
		
		
		\bibitem{R} Rudin, W., \textit{Real and complex analysis}, McGraw-Hili Book Co., Singapore,  ISBN 0-07-100276-6. 
		
		
%		\bibitem{R70} Rockafellar, R. T., \textit{Convex Analysis}, Princeton University Press, 1970.
		
		\bibitem{S1} Sain, D., \textit{ On the norm attainment set of a bounded linear operator}, J. Math. Anal. Appl., \textbf{457}	(2018) 67–76.
		
		
		\bibitem{S2} Sain, D.,\textit{Birkhoff-James orthogonality of linear operators on finite-dimensional Banach spaces}, J. Math. Anal. Appl., \textbf{447} (2017) 860-866.
		
		\bibitem{SMP} Sain, D., Mal, A., Paul, K., \textit{Some remarks on Birkhoff–James orthogonality of linear	operators}, Expo. Math., \textbf{38} (2020) 138-147.
		
			\bibitem{SP} Sain, D., Paul, K., \textit{Operator norm attainment and inner product spaces}, Linear Algebra Appl., \textbf{439}(2013) 2448-SP2452.
		
		
		\bibitem{SPM} Sain, D., Paul, K., Mal, A., \textit{A complete characterization of Birkhoff-James orthogonality in infinite dimensional normed space} J. Operator Theory \textbf{80}(2018) 399-413.
		
		\bibitem{SRBB} Sain, D., Roy, S., Bagchi, S., Balestro, V., \textit{A study of symmetric points in Banach spaces}, Linear Multilinear Alebra, \textbf{70}(2022) 888-898.
		
		
	
		
		
		\bibitem{T} Tur$\check{s}$ek, A., \textit{A remark on orthogonality and symmetry of operators in $\mathcal{B}(\mathcal{H})$}, Linear Algebra Appl., \textbf{535}(2017) 141-150.

%\bibitem{S2} Sain, D., \textit{On the norm attainment set of a bounded linear operator}, J. Math. Anal. Appl., \textbf{457}(2018) 67-76.
	\end{thebibliography}
\end{document}